\documentclass[british]{article}
\usepackage{ae,aecompl}
\usepackage[T1]{fontenc}
\usepackage[latin9]{inputenc}
\usepackage{geometry}
\usepackage{color}
\usepackage{amsmath}
\usepackage{amsthm}
\usepackage{amssymb}
\PassOptionsToPackage{normalem}{ulem}
\usepackage{ulem}

\makeatletter
\theoremstyle{plain}
\newtheorem{thm}{\protect\theoremname}
\theoremstyle{definition}
\newtheorem{defn}[thm]{\protect\definitionname}
\newtheorem{example}[thm]{Examples}
\newtheorem{ass}[thm]{Assumptions}
\theoremstyle{remark}
\newtheorem{rem}[thm]{\protect\remarkname}
\theoremstyle{plain}
\newtheorem{lem}[thm]{\protect\lemmaname}
\theoremstyle{plain}
\newtheorem{prop}[thm]{\protect\propositionname}
\theoremstyle{plain}
\newtheorem{cor}[thm]{\protect\corollaryname}

\newcommand{\eh}{\frac{1}{2}}
\newcommand{\IP}{\mathbb{P}}
\newcommand{\IE}{\mathbb{E}}
\newcommand{\uX}{\underline{X}}
\newcommand{\ux}{\underline{x}}
\newcommand{\uN}{\underline{N}}
\newcommand{\om}{\overline{m}}
\newcommand{\IR}{\mathbb{R}}
\newcommand{\supp}{\textnormal{supp}}
\date{}

\makeatother

\usepackage{babel}
\providecommand{\corollaryname}{Corollary}
\providecommand{\definitionname}{Definition}
\providecommand{\lemmaname}{Lemma}
\providecommand{\propositionname}{Proposition}
\providecommand{\remarkname}{Remark}
\providecommand{\theoremname}{Theorem}

\begin{document}
\title{Collective Bias Models in Two-Tier Voting Systems and the Democracy
Deficit}
\author{Werner Kirsch\thanks{FernUniversit\"at in Hagen, Germany, werner.kirsch@fernuni-hagen.de}
\;and Gabor Toth\thanks{IIMAS-UNAM, Mexico City, Mexico, gabor.toth@iimas.unam.mx}}
\maketitle
\begin{abstract}
\noindent We analyse optimal voting weights in two-tier voting systems. In our model,
the overall population (or union) is split in groups (or member states) of different 
sizes. The individuals comprising the overall population constitute the first tier, and the council is the second tier. Each group has a representative in the council that casts votes on their behalf. By `optimal weights', we mean voting weights in the council which
minimise the democracy deficit, i.e.$\!\,$ the expected deviation of the council vote
from a (hypothetical) popular vote. 

\noindent We assume that the voters within each group interact via what we 
call a local collective bias or common belief (through tradition, common values, strong religious beliefs, etc.). We allow in addition an interaction across
group borders via a global bias. Thus, the voting behaviour of each voter depends on the behaviour
of all other voters. This correlation may be stronger between voters in the same group, but is in general
not zero for voters in different groups.

\noindent We call the respective voting measure a Collective Bias Model (CBM). The `simple CBM' introduced
in \cite{HOEC} and in particular the Impartial Culture and the Impartial Anonymous Culture are special cases of our general model.

\noindent We compute the optimal weights in the large population limit. Those optimal weights are
unique as long as there is no `complete' correlation between the groups. In this case, we obtain optimal weights which are the sum of a common constant equal for all groups and a summand which is proportional to the population of each group. If the correlation between
voters in different groups is extremely strong, then the optimal weights are not unique. 
In fact, in this case, the weights are essentially arbitrary. We also analyse the conditions under which the optimal weights are negative, thus making it impossible to reach the theoretical minimum of the democracy deficit. This is a new aspect of the model owed to the correlation between votes belonging to different groups.

\end{abstract}


\section{Introduction}

We study voting in two-tier voting systems.  Suppose the population of a state or union of states is subdivided
into $M$ groups (member states for example). Each group sends a representative
to a council which makes decisions for the union. The representatives cast their vote (`aye' or `nay') according to
the majority (or to what they believe is the majority) in their respective  group. Since the groups
may differ in size, it is natural to assign different voting
weights to the representatives, reflecting the size of the respective group. When a parliament such as a the House of Representatives in the U.S. is elected, usually the country is subdivided into a number of districts of roughly equal population, each of which votes on a representative for a single seat. This procedure is feasible within a country but may not be possible in other situations. Even in the U.S., no effort has been made to divide the states and reassemble them into districts of roughly equal size so that each of them could have the same number of senators without giving rise to questions whether that is the `right' way to determine the number of senators. It is even less likely that sovereign countries -- such as the members of the United Nations or the European Union -- would be willing to submit to being divided into districts of equal size. Thus, it is not possible in practice to circumvent the question of how to assign voting weights to groups of different sizes.

To determine these weights is the problem of `optimal' weights. How should the weights be assigned? One objective
studied in the literature is to minimise the democracy deficit, i.e.\!
the deviation of the council vote from a hypothetical
referendum across the entire population. The democracy deficit was first studied for binary voting (the same setting which is considered in the present article) by Felsenthal and Machover \cite{FM1999}. Later on, it was also analysed in other settings by several authors (see e.g.\! \cite{FelsenthalM,HOEC,SlomZy,KiLa,MN2012,Toth}). Other notions of optimal weights are based on welfare considerations or the criterion of equalising the influence of all voters belonging to the overall population. In the latter vein, we find the seminal article by Penrose \cite{Pen1946}, where the square root law was first established as the assignation rule for voting weights that equalises the probability of each voter's being decisive in a two-tier voting system under the assumption of stochastically independent voting. Other contributions to the study of optimal voting weights under welfare and influence frameworks can be found in \cite{MN2012,BB2007,KMTL2013,KMN2017}. Correlated voting across groups was also analysed by Kaniovski and Zaigraev in \cite{KanZai2009}.

Suppose the overall population is of size $N$, whereas the group size is
 $N_{\lambda}$, where the subindex $\lambda$ stands for
the group $\lambda\in\left\{ 1,\ldots,M\right\} $ . Let the two voting alternatives
be encoded as $\pm1$, $+1$ for `aye' and $-1$
for `nay'. The vote of voter $i\in\left\{ 1,\ldots,N_{\lambda}\right\} $ in group $\lambda $ will be denoted by $X_{\lambda i}$.

\begin{defn}
For each group $\lambda$, we define the voting margin $S_{\lambda}:=\sum_{i=1}^{N_{\lambda}}X_{\lambda i}$.
The overall voting margin is $S:=\sum_{\lambda=1}^{M}S_{\lambda}$.
\end{defn}

Each group casts a vote in the council:
\begin{defn}
The council vote of group $\lambda$ is given by
\[
\chi_{\lambda}:=\begin{cases}\,
\phantom{-}1, & \text{if }S_{\lambda}>0,\\
\,-1, & \text{otherwise.}
\end{cases}
\]
\end{defn}

The (representative of) group $\lambda$ votes `aye' if there is a majority in group $\lambda $ on the issue in question. Each group $\lambda$ is assigned a weight $w_\lambda$. The weighted sum $\sum_{\lambda=1}^{M}w_{\lambda}\chi_{\lambda}$
is the council vote. The council vote is in favour of a proposal if
$\sum_{\lambda=1}^{M}w_{\lambda}\chi_{\lambda}>0$. Weights $w_{1},\ldots, w_{M} $ together with a relative quota $q\in(0,1)$ constitute
a weighted voting system for the council, in which a coalition $C\subset \{ 1,2,\ldots,M \}$ is winning if
\begin{align*}
   \sum_{i\in C}\,w_{i}~>~q\;\sum_{i=1}^{M}\,w_{i}.
\end{align*}
We will exclusively consider the majority rule with $q=1/2$ in this article.

It is reasonable to choose the voting weights $w_{\lambda}$ in the council in such a way, that
the raw democracy deficit
\begin{align*}
   \left|\,S-\sum_{\lambda=1}^{M}\,w_{\lambda}\,\chi_{\lambda}\,\right|
\end{align*}
is as small as possible. For a given set of weights, each configuration of all $N$ votes induces a certain raw democracy deficit. It is immediately clear that in general there is no choice of weights which makes
this variable small uniformly over all possible distributions of Yes-No-votes across the overall population. All we can hope for is to
make it small `on average', more precisely we try to minimise the expected quadratic deviation
of $\sum_{\lambda=1}^M w_{\lambda}\chi_{\lambda}$ from $S$.

To follow this approach, we have to clarify what we mean by
`expected' deviation, i.e.\! there has to be some notion of randomness underlying the voting procedure.

While the votes cast are assumed to be deterministic and rational, obeying the
voters' preferences which we do not model explicitly, the proposal
put before them is assumed to be unpredictable, i.e.\! random. Since each yes/no
question can be posed in two opposite ways, one to which a given voter
would respond `aye' and one to which they would respond `nay', it is reasonable to assume that each voter votes `aye' with
the same probability they vote `nay'.

This leads us to the following definition:

\begin{defn}
A \emph{voting measure} is a probability measure $\IP $ on the space
of voting configurations $\left\{ -1,1\right\} ^{N}= \prod_{\lambda=1}^{M}\, \{ -1,1 \}^{N_{\lambda}}$ with the symmetry
property
\begin{align}
\mathbb{P}\left(X_{11}=x_{11},\ldots,X_{MN_{M}}=x_{MN_{M}}\right)~
=~\mathbb{P}\left(X_{11}=-x_{11},\ldots,X_{MN_{M}}=-x_{MN_{M}}\right)\label{eq:symmetry}
\end{align}
for all voting configurations $\left(x_{11},\ldots,x_{MN_{M}}\right)\in\left\{ -1,1\right\} ^{N}$.

By $\mathbb{E}$ we will denote the expectation with respect to $\mathbb{P}$.
\end{defn}

The simplest and widely used voting measure is the $N$-fold product of the measures
\begin{align*}
   P_{0}(1)~=~P_{0}(-1)~=\frac{1}{2}
\end{align*}
which models independence between all the individual votes $X_{\lambda i}$. In this
special case, known as the \emph{Impartial Culture} (see e.g. \cite{GK1968}, \cite{GehrleinL}, or \cite{KurzMN}), we have
\begin{align*}
   \mathbb{P}\left(X_{11}=x_{11},\ldots,X_{MN_{M}}=x_{MN_{M}}\right)~
   =~\prod_{\lambda=1}^{M}\prod_{i=1}^{N_{\lambda}}\,P_{0}(X_{\lambda i}=x_{\lambda i})~=~\frac{1}{2^{N}}.
\end{align*}

This article treats the class of voting measures called the \emph{collective bias model} (or common belief model, CBM) which extends
the Impartial Culture considerably by allowing correlations both between voters in the same group
as well as correlations across group borders. We introduce and discuss the CBM in Section \ref{sec:CBM}.

Once a voting measure is given, the quantities $X_{\lambda i}$, $S_{\lambda}$, $\chi_{\lambda}$, and the raw democracy deficit are random variables defined on the same probability space $\{-1,1\}^{N}$.

Now we can define the concept of democracy deficit which is a measure of how well the council votes follow the public opinion:

\begin{defn} \label{def:democracy-deficit}The \emph{democracy deficit} given a voting
measure $\mathbb{P}$ and a set of weights $w_{1},\ldots,w_{M}$ is
defined by
\begin{align*}
\Delta_{1}~=\Delta_{1}(w_{1},\ldots,w_{M})~:=~\mathbb{E}\left[\left(S
-\sum_{\lambda=1}^{M}w_{\lambda}\chi_{\lambda}\right)^{2}\right]\,.
\end{align*}

We call $(w_{1},\ldots, w_{M}) $ \emph{optimal weights} if they minimise the democracy deficit, i.e.
\begin{align*}
   \Delta_{1}(w_{1},\ldots,w_{M})~=~\min_{(v_{1},\ldots,v_{M})\in\mathbb{R}^{M}}\;\Delta_{1}(v_{1},\ldots,v_{M}).
\end{align*}
\end{defn}

Note that the democracy deficit depends on the voting measure. It is also worth pointing out that the democracy deficit is a differentiable function of the council weights. This facilitates the analysis required to find the optimal weights. 

Instead of minimising the democracy deficit, we could ask the question of how to minimise the probability that the binary council decision differs from the decision made by a referendum. This would be a less strict criterion in the sense that for a favourable public opinion of 51\%, a 51\% percent vote in the council and a 100\% vote would both be considered equally good. However, one could argue that a 100\% vote in the council would not be a good representation of public opinion. In fact, the 49\% minority might feel they are not represented in the council at all, giving rise to populist anti-elite sentiment among the minority. We argue that adjusting the voting outcomes in the council in such a way that they follow the popular opinion as closely as possible is a worthwhile goal.

If we multiply each weight by the same positive constant and keep the relative quota $q$ fixed, we obtain an equivalent voting system. If the weights $w_{\lambda}$ minimise the democracy deficit $\Delta_{1} $, then
the (equivalent)
weights $ \frac{ w_{\lambda}}{\sigma}$ minimise the `renormalised' democracy deficit
$\Delta_{\sigma} $ defined by
\begin{align*}
   \Delta_{\sigma}~=~\Delta_{\sigma}(v_{1},\ldots,v_{M})~:=~\mathbb{E}\left[\left(\frac{S}{\sigma}
-\sum_{\lambda=1}^{M}v_{\lambda}\chi_{\lambda}\right)^{2}\right]\,.
\end{align*}

It is, therefore, irrelevant whether we minimise $\Delta_{1}$ or $\Delta_{\sigma} $ as long as $\sigma>0 $.
In this article, we will compute optimal weights as $N$ tends to infinity. As a rule, in this limit the minimising weights
for $\Delta_{1}$ will also tend to infinity, it is therefore useful to minimise $\Delta_{\sigma}$ with an $N$-dependent $\sigma $ to keep the weights bounded. A particularly convenient choice is to
normalize the weights $w_{\lambda}$ in such a way that $\sum_\lambda w_{\lambda}=1$.

The rest of this paper is organised as follows: as a first step, in Section \ref{warm-up}, we recall the CBMs with independent groups studied in the past and give an example of a CBM with correlated voting across group boundaries. Then, we formally define the CBM and give several more examples in Section \ref{sec:CBM}. In Section \ref{sec:demdef}, we discuss the problem of determining the optimal weights in order to minimise the democracy deficit. Section \ref{sec:asymp} contains the results concerning the large population behaviour of  CBMs. In Sections \ref{sec:optw} and \ref{sec:optwtc}, we calculate the optimal weights in the large population limit. Then, Section \ref{sec:specmod} discusses the optimal weights for some specific models introduced earlier, such as additive and multiplicative models. Sections \ref{sec:optw} to \ref{sec:specmod} contain the main results of this article concerning the optimal weights for a large set of CBMs. The second part of our analysis of optimal weights concerns their non-negativity. Under independence of the groups, the optimal weights are always strictly positive. Thus, this is a new aspect owed entirely to the relaxation of the independence assumption and not previously analysed in the literature. Section \ref{sec:nonneg} deals with the problem of negative optimal weights and conditions that rule them out. Section \ref{sec:extensions} presents an extension of the CBM with non-identical group bias distributions. Finally, Section \ref{sec:conclusion} concludes the paper.

\section{A Warm-Up}\label{warm-up}

Before defining the CBM in its full generality, we first recall the
CBMs treated in the past, where the biases in different groups are
independent of each other, and hence the voters belonging to different
groups act independently. Then, we give an example of a CBM with correlated
groups. It is our hope that the informal description of these special
cases before defining and analysing the CBM in its full generality
will make the model and the article more accessible to a wider range
of readers.

In \cite{HOEC}, one of us introduced the CBM with
groups still being independent. To distinguish it from the generalisation we are going to introduce
below, we refer to the CBM with independent groups as the \emph{simple CBM} for the rest of this paper.
Let $T_{1},\ldots,T_{M}$ be a collection of independent and identically
distributed random variables with support in $\left[-1,1\right]$.
We will refer to a realisation of each of these variables as $t_{\lambda}$.
Conditionally on $t_{\lambda}$, the voters in group $\lambda$ vote
independently of each other and the probability of a `yes' vote
is $\frac{1+t_{\lambda}}{2}$. Thus, a positive bias makes a $+1$
vote more likely from each of the voters. Since within each group
all voters are subject to the same bias, there is, in fact, positive
correlation between votes. There is, however, no correlation between
votes in different groups. The group bias reflects some real-world
influence on the voters' decisions, such as cultural norms, institutions
such as organised religions, or more recently social media and influencers, specific to each group.
Suppose each of these entities has some stance on each issue that
can be put to vote and these opinions aggregate to some public bias.
This bias can be quantified: a value of $-1$ or close to it reflects
a strong rejection; a value around 0 means neutrality or indifference;
a positive value close to $+1$ reflects strong support of an issue.
The bias affects the voting outcomes in such a way that, provided
the population is large enough, a strong negative bias will result
in a large negative vote $S_{\lambda}$. Similarly, a strong positive
bias induces a large positive vote, and the absence of a substantial
bias causes a small absolute voting margin $\left|S_{\lambda}\right|$,
i.e.\! a voting outcome close to a tie, because the individual votes are nearly
independent.

In order to obtain the distribution of votes in the overall population,
we have to average out the votes given all possible values of the
group bias variables. Due to the independence of the group bias variables
$T_{1},\ldots,T_{M}$, we can factor the probabilities and obtain
for each voting configuration $\left(x_{11},\ldots,x_{MN_{M}}\right)\in\left\{ -1,1\right\} ^{N}$
\[
\IP\left(X_{11}=x_{11},\ldots,X_{MN_{M}}=x_{MN_{M}}\right)=\prod_{\lambda=1}^{M}\IP\left(X_{\lambda1}=x_{\lambda1},\ldots,X_{\lambda N_{\lambda}}=x_{\lambda N_{\lambda}}\right)
\]
because the groups are independent. In accordance with the verbal
description given in the last paragraph, the probabilities for each
group's voting configuration can be expressed as follows:
\[
\IP\left(X_{\lambda1}=x_{\lambda1},\ldots,X_{\lambda N_{\lambda}}=x_{\lambda N_{\lambda}}\right)=E\left(\prod_{i=1}^{N_{\lambda}}P_{T_{\lambda}}\left(X_{\lambda i}=x_{\lambda i}\right)\right),
\]
where $P_{T_{\lambda}}$ is the probability measure on $\left\{ -1,1\right\} $
with $P_{T_{\lambda}}\left(1\right)=\frac{1+T_{\lambda}}{2}$ and
the expectation $E$ is taken with respect to the distribution of the group
bias variable $T_{\lambda}$. To recapitulate, we obtain the probability
of a voting configuration by taking the (random) probabilities $P_{T_{\lambda}}\left(X_{\lambda i}=x_{\lambda i}\right)$
for each individual vote and multiplying them together for all voters
belonging to group $\lambda$. This multiplicative form of the probabilities
is due to the conditional independence of the individual votes in
the group for a given realisation of the group bias variable $T_{\lambda}=t_{\lambda}$.
Next, we take the expectation of this product of probabilities over
all possible realisations of $T_{\lambda}$. This gives us the probability
of the voting configuration $\left(x_{\lambda1},\ldots,x_{\lambda N_{\lambda}}\right)$
in group $\lambda$. Finally, multiplying the probabilities of each
group's voting configuration yields the probability of the overall
voting configuration $\left(x_{11},\ldots,x_{MN_{M}}\right)$ under
the simple CBM.

For the simple CBM, we can calculate the optimal weights in the council
(see \cite{HOEC} for this result). These turn out to be proportional
to the expected absolute value of the group voting margins, $w_{\lambda}\propto\IE\left(\left|S_{\lambda}\right|\right)$.
It is known that as each group's population diverges to infinity,
we have $\IE\left(\left|S_{\lambda}\right|\right)/N_{\lambda}\rightarrow E\left|T_{\lambda}\right|$.
The latter expectation is a characteristic of the underlying distribution
of the bias variables $T_{\lambda}$. This implies (under the mild
assumption $E\left|T_{\lambda}\right|>0$) that in the large population
limit the optimal weights in the council are proportional to each
group's population, i.e.\! $w_{\lambda}=CN_{\lambda}$, with the
same positive multiplicative constant $C$ for each group.

Thus, the simple CBM yields the recommendation of assigning each group
a number of votes in the council which is proportional to its population.
This stands in contrast to Penrose's square root law, which prescribes
weights proportional to the square root of the population instead
of the population itself. Evidently, proportionality favours larger
groups at the expense of smaller ones and vice-versa for the square
root law.

Next, we consider a model with bias variables $T_{1},\ldots,T_{M}$
which are correlated, thus inducing correlated voting across groups
boundaries. Let $Z$ be a uniformly distributed random variable on
the interval $\left[-1/2,1/2\right]$. We will call $Z$ the global
bias variable. Let $Y_{1},\ldots,Y_{M}$ be i.i.d.\! copies of $Z$.
We define the group bias variables $T_{1},\ldots,T_{M}$ by setting
$T_{\lambda}:=Y_{\lambda}+Z$ for each group $\lambda$. This is a
special case of what we will call an `additive model' in later sections.

In an additive model, the global bias is modified by a group-specific bias which may reinforce or counteract the global bias. By assuming all these variables are identically distributed, we assign equal influence to the global bias and group-specific attitudes. Of course, it is also possible to assume a stronger global or local bias, a topic we will explore in Section \ref{sec:additive}.

Even though $Z,Y_{1},\ldots,Y_{M}$ are independent, due to the addition
of the same variable $Z$ in the definition of each $T_{\lambda}$,
the group bias variables $T_{\lambda}$ are \emph{not} independent.
Given a realisation $t_{\lambda}=y_{\lambda}+z$, the individual votes
in group $\lambda$ each turn out positive with probability $\frac{1+t_{\lambda}}{2}$.
Contrary to the simple CBM, we cannot factor the probabilities of
the overall voting configuration into the probabilities of the group
voting configurations. Instead, the probabilities can be expressed
as
\begin{align*}
\IP\left(X_{11}=x_{11},\ldots,X_{MN_{M}}=x_{MN_{M}}\right) & =E\left(\prod_{\lambda=1}^{M}\prod_{i=1}^{N_{\lambda}}P_{T_{\lambda}}\left(X_{\lambda i}=x_{\lambda i}\right)\right),
\end{align*}
where the expectation $E$ is taken with respect to the distributions of
$Z,Y_{1},\ldots,Y_{M}$. As the terms $P_{T_{\lambda}}\left(X_{\lambda i}=x_{\lambda i}\right)$
each depend on two different random variables, $Y_{\lambda}$ and
$Z$, there is no way to factor the expectation above.

In order to minimise the democracy deficit, we have to solve the linear
equation system given by \eqref{eq:LES}. We omit the calculations and refer the
reader to Section \ref{sec:additive}, where we will analyse a more general additive
CBM with uniformly distributed bias variables. By Theorem \ref{opt_weights_uniform}, the
optimal weight for each group $\lambda$ is asymptotically given by
\[
w_{\lambda}=D+C\frac{N_{\lambda}}{N},
\]
where we have simplified and normalised the weights. The positive
constants $C$ and $D$ are common for all groups. Contrary to the
simple CBM with independent groups, in this correlated example, we
have a summand which is proportional to the size of the group, but
we also have a constant summand $D$ which is the same for each group
and hence independent of the group's size. This is qualitatively the
same formula as the one employed for the composition of the U.S.\!
Electoral College, where $D$ equals 2, the number of senators for each state,
and $C=435$ is the number of representatives, of which each state
receives a number roughly proportional to its population. The new
feature of the general CBM compared to the simple CBM concerning the
problem of optimal council weights is the presence of the constant
term $D$. This functional form for the optimal weights applies not
just to the special case discussed in this section but in general
to CBMs with correlated groups. See Theorem \ref{thm:asweights} for the general result.
As far as the authors of this article know, this is the first theoretical
justification of the formula that determines the number of electors
for each state in the U.S.\! Electoral College. We want to emphasise
that it is the correlation between votes in different groups that
induces the constant term $D$ in the formula for the optimal weights.
No matter what distribution we choose for the group bias variables
$T_{\lambda}$ in a simple CBM \textendash{} even allowing for $T_{\lambda}$
that follow different distributions \textendash{} the constant $D$
always equals 0. The formula $w_{\lambda}=D+C\frac{N_{\lambda}}{N}$
is more favourable to small groups than proportionality. However,
the square root law is even better for small groups in most cases,
the exception being when the difference in size between small and
large groups is minuscule.

\section{The Collective Bias Model}\label{sec:CBM}

We recall from the last section that in the simple CBM, the votes $X_{\lambda i}$ within a group $\lambda $ are correlated via a random
variable $T_{\lambda}$ with values in $[-1,1] $, the local `collective bias'. The random variables $T_{\lambda}$ model the influence
of a cultural tradition in the respective group or the leverage of a strong political party (or religious group, etc.) within the group $\lambda $. It is this central influence, which affects all voters within a given group equally, which induces positive correlation between votes within each group. Aside from this central influence, the voters make up their own minds. This is in contrast to models with interactions between voters such as those inspired by spin models from statistical mechanics, e.g.\! mean-field models (see \cite{KT_opt_weights_CW} for a discussion of a mean-field model and the determination of the optimal weights). In the simple CBM, there is no correlation between votes in different groups, only correlation within groups.

In what follows, we will define the simple CBM in the same measure-theoretic language we will also employ for the general CBM. Given the bias $T_{\lambda}=t_{\lambda}$, the per capita voting margin $S_\lambda/N_\lambda$ inside group $\lambda$ fluctuates around $t_{\lambda}$.
More precisely, suppose the bias variable $T_{\lambda} $ is distributed according to the probability
measure $\rho $ on $[-1,1] $. Then the simple CBM for the group $\lambda $ is given by
\begin{align}\label{eq:simpleCBM}
   \mathbb{P}\Big(X_{\lambda 1}=x_{1}, X_{\lambda 2}=x_{2},\ldots,X_{\lambda N_{\lambda}}=x_{N_{\lambda}}\Big)~=~
   \int P_{t}\big(x_1,\ldots, x_{N_{\lambda}}\big)\rho(\textup{d}t),
\end{align}
where
\begin{align*}
& P_{t}(x)~=~\left\{
             \begin{array}{ll}
               \eh(1+t), & \hbox{for $x=1$,} \\[2mm]
               \eh(1-t), & \hbox{for $x=-1$,}
             \end{array}
           \right.\\
\text{and}\quad & P_{t}\big(x_1,\ldots, x_{N_{\lambda}}\big)~
=~P_{t}(x_{1})\,P_{t}(x_{2})\,\cdots\,P_{t}(x_{N_{\lambda}}).
\end{align*}

By $E_{t} $ we denote the expectation with respect to $P_{t} $. The definition of $P_{t}$ implies that $E_{t}(X)=t$.
We call $\rho $ the local bias measure of group $\lambda$.

We remark that, due to de Finetti's Theorem\footnote{De Finetti's Theorem states that an infinite sequence of exchangeable random variables can be represented as a mixture of i.i.d.\! random variables. The mixture is specified by a probability measure referred to as a de Finetti measure. De Finetti's Theorem has been considerably generalised. See e.g.\! \cite{DF1987}.}, the simple CBM is the most general voting measure that is `anonymous' in the sense that reordering the voters leaves the measure unchanged (see \cite{Klenke} or \cite{KiFinetti}).

The `Impartial Anonymous Culture', which underlies the Shapley-Shubik power index \cite{Straffin} (see also \cite{GehrleinL} or \cite{KurzMN}), is a particular case of \eqref{eq:simpleCBM} where $\rho $ is the
uniform distribution on $[-1,1] $. The Impartial Culture is another special case for which $\rho=\delta_{0}$, the Dirac measure\footnote{The Dirac measure (or point mass) at $x\in \IR^M$, $\delta_x$, is a probability measure which assigns any set $A\subset \IR^M$ the probability 1 if $x\in A$ and 0 otherwise.} at $t=0$.

In the simple CBM, the voting results in different groups are  independent, so the corresponding voting measure on $\prod_{\lambda=1}^{M} \{ -1,1 \}^{N_{\lambda}} $ is given by the product of the probabilities \eqref{eq:simpleCBM}.
\begin{align*}
   \mathbb{P}\left(\uX_{1}=\ux_{1},\ldots,\uX_{M}=\ux_{M}\right)~
=~\int P_{t_1}\big(\ux_{1}\big)\;\rho(\textup{d}t_{1})~\cdots
~\int P_{t_M}\big(\ux_{M}\big)\;\rho(\textup{d}t_{M}),
\end{align*}
where $\underline{X}_{\lambda}=(X_{\lambda 1},\ldots,X_{\lambda N_{\lambda}})$ and similarly for $\underline{x}_{\lambda}$.

In this paper, we study the \emph{generalised collective bias model} (CBMs with dependence across group boundaries were first analysed in \cite{Toth}). In this model, there
is an additional \emph{global} bias
variable $Z$ with values in $[-1,1] $ and with distribution $\mu $. The global bias $Z$ influences each of the groups in a similar way. This is implemented in the model by allowing the  local bias measure $\rho $ to depend on the
value $Z=z$. More precisely, the (generalised) collective bias model is given by:

\begin{defn}\label{def:CBM}
   Suppose $\mu $ is a probability measure on $[-1,1] $ and for every $z\in[-1,1]$ there is a
probability measure $\rho^{z}$ on $[-1,1] $. Then we define the probability measure $\IP_{\mu \rho}$
on $\{ -1,1 \}^{N}=\prod_{\lambda=1}^{M}\{ -1,1 \}^{N_{\lambda}}$ by
\begin{align}
   \IP\big(\ux_{1},\ux_{2},\ldots,\ux_{M}\big)~
=~&\int  \left( \int P_{t_{1}}(\ux_{1})\,\rho^{z}(\textup{d}t_{1})
\cdots\int P_{t_{M}}(\ux_{M})
\,\rho^{z}(\textup{d}t_{M}) \right)\;\mu(\textup{d}z),\label{eq:genCBM}
\end{align}
where $\ux_{\lambda}\in\{ -1,1 \}^{N_{\lambda}}$.

We call $\IP_{\mu\rho} $ the \emph{collective bias measure} with \emph{global bias measure} $\mu$ and \emph{local bias measure} $\rho = \rho^{z} $ or the CBM($\mu,\rho $) for short. If $\mu $ and $\rho $ are clear from the context, we simply write $\IP$ instead of $\IP_{\mu\rho}$.
\end{defn}

\begin{rem}\label{rem:kernel}
Technically speaking, $\rho^{z} $ is a stochastic kernel (see e.g. \cite{Klenke}), i.e.:
\begin{enumerate}
\item For every $z\in[-1,1]$, the quantity $\rho^{z}$ is a probability measure on $[-1,1] $.
\item For every Borel set $A\subset[-1,1]$, the function $z\mapsto \rho^{z}A$ is measurable.
\end{enumerate}
\end{rem}
We could allow the kernels $\rho^{z}$ to depend on
the group $\lambda $, and in Section \ref{sec:extensions} we will come back to this generalisation,
but for the moment we take the same local bias measure for all groups.

To ensure that $\IP_{\mu\rho}$ is a \emph{voting measure}, i.e.$\!\,$ to satisfy \eqref{eq:symmetry},
we assume the following sufficient condition in what follows:
\begin{ass}\label{ass:sym}
\begin{enumerate}
\item $\mu$ is symmetric, i.e. $\mu A=\mu(-A)$,
\item for all $z\in\left[-1,1\right]$, the distributions
$\rho^{z}$ satisfy $\rho^{z}A=\rho^{-z}\left(-A\right)$ for all measurable sets $A\subset\left[-1,1\right]$.
\end{enumerate}
\end{ass}

The general framework of a CBM is given by a set
of bias random variables that represent some cultural or political
influence that acts on all voters. There is a \emph{global bias variable}
$Z$ with distribution $\mu$ which induces correlation between voters
of different groups. Furthermore, there is a \emph{local bias variable} $T_{\lambda}$
for each group. Its conditional distribution given the realisation
$Z=z$ is $\rho^{z}$. The group bias variable $T_{\lambda}$
induces correlation between the voters belonging to that group. The result is correlated voting across group boundaries, as a rule with stronger
correlation within each group to account for shared culture and preferences.

Conditionally on the realisations of $Z=z$ according to $\mu$
and $T_{\lambda}=t_{\lambda}$ according to $\rho^{z}$, all
voters in group $\lambda$ cast their vote independently, with a probability
of voting `aye' equal to $\frac{1+t_{\lambda}}{2}$. Hence,
a value $t_{\lambda}=1$ implies that all voters belonging to
group $\lambda$ vote `aye' almost surely. Similarly, $t_{\lambda}=-1$
implies all vote `nay' almost surely. $t_{\lambda}=0$ means
there is no bias, and all voters in the group vote independently with probability $\tfrac{1}{2}$ for `aye' (and the same probability for `nay').

\begin{example}\label{ex:1}
{ }\quad We discuss various examples (or classes of examples) of CBMs.
\begin{enumerate}
   \item If the measures $\rho^{z} $ are independent of $z$, then the (generalised)
    CBM reduces to the simple CBM. The Impartial Anonymous Culture is a particular case of this class of examples.
    \item If $\rho^{z}=\delta_{0} $, then all random variables
    $X_{\lambda i}$ are independent reflecting Impartial Culture.
\item If $\rho^{z}=\delta_{z}$, then we have a simple CBM for the \emph{union}, i.e.\! for \emph{all} $X_{\lambda i}$.
\item \label{ex:add} In the class of \emph{additive models}, the `total bias' prevailing within each group $T_\lambda$ is the \emph{sum} of the global bias variable $Z$ and a local or group bias modifier variable $Y_{\lambda}$, i.e. $T_\lambda=Z+Y_\lambda$.
Assume the bias modifiers $Y_\lambda$ are independent and identically distributed according to a fixed symmetric probability measure $\rho$.
Then, for each realisation $Z=z$, the local measure $\rho^{z}$ is given by
\begin{align*}
   \rho^{z}[a,b]~=~\rho[a-z,b-z].
\end{align*}
So for this model class we have
\begin{align}
\IP\big(\ux_{1},\ux_{2},\ldots,\ux_{M}\big)~=~\int  \left( \int P_{z+y_{1}}(\ux_{1})\,\rho(\textup{d}y_{1})
\cdots\int P_{z+y_{M}}(\ux_{M})
\,\rho(\textup{d}y_{M}) \right)\;\mu(\textup{d}z). \label{eq:addCBM}
\end{align}
To ensure that $t_{\lambda}=z+y_\lambda\in[-1,1]$ we assume that $\supp\,\mu\,+\,\supp\,\rho\;\subset[-1,1]$, where $\supp\, \mu$ stands for the support of the measure $\mu$. This kind of additive CBM was first introduced and analysed in Section 4.2 of \cite{Toth}. Additive models are discussed in more detail in Section \ref{sec:additive}.
\item For a particular example of the additive model which we are going to discuss in some detail,
 we choose $\mu $ and $\rho $ as the uniform distribution on $[-g,g] $ and on $[-\ell,\ell] $, respectively, with $0<g,\ell$ and $g+\ell\leq 1$.

In this case, the additive CBM-measure is given by
\begin{align*}
   \frac{1}{2g}\int_{-g}^{+g}\;\left(\,\frac{1}{2\ell}\int_{z-\ell}^{z+\ell} P_{t_{1}}(\underline{x}_{1})\,\textup{d}t_{1}\cdots
\frac{1}{2\ell}\int_{z-\ell}^{z+\ell} P_{t_{M}}(\underline{x}_{M})\,\textup{d}t_{M}\,\right)\,\textup{d}z.
\end{align*}
This example may be considered a `hierarchical' version of Impartial Anonymous Culture.
\item\label{ex:mult} In the class of \emph{multiplicative models}, the total bias is the \emph{product} of the global bias variable $Z$ and the group bias modifier variable $Y_{\lambda}$, i.e. $T_\lambda=ZY_\lambda$. We assume the $Y_\lambda$ are independent and identically distributed according to a fixed probability measure $\rho$. Then the local measure is $\rho^0=\delta_0$ if $Z=0$, and for $Z=z\neq0$,
\begin{align*}
   \rho^{z}[a,b]~=~\rho\left[\frac{a}{z}\wedge\frac{b}{z},\frac{a}{z}\vee\frac{b}{z}\right].
\end{align*}
Above, we used the notation $a\vee b:=\max\{a,b\}$ and $a\wedge b:=\min\{a,b\}$ for all real numbers $a$ and $b$.

This kind of multiplicative CBM was first introduced and analysed in Section 4.1 of \cite{Toth}. We discuss the multiplicative model in Section \ref{sec:mult}.
\item In the CBM($\mu,\rho$), the measure $\rho $ must have support in $[-1,1] $. Above, we assumed
without loss of generality the same for the measure $\mu $. In the following example, it is more convenient to have more freedom in the choice of $\mu $.

Suppose that $\rho^{z} $ is the beta distribution $B(z,z,-1,1) $, i.e. $\rho^z$ has the density
\begin{align*}
   f_{z}(x):=\frac{\Gamma(2z)}{\Gamma(z)^{2}\; 2^{2z-1}}\;(1+x)^{z-1}\,(1-x)^{z-1}
\end{align*}
for $x\in[-1,1]$, where $\Gamma $ is the Gamma function. For $\mu$, we can take any probability distribution on $(0,\infty) $. Note that the symmetry condition \eqref{eq:symmetry} is satisfied.

For large $z$, the measures $\rho^{z} $ are more and more concentrated around $0$. For $z=1$, the measure
$\rho^{z} $ is the uniform distribution, and for small $z>0,$\;$\rho^{z} $ is more and more concentrated near the extreme positions $+1$ and $-1$. The measures $\rho^{z} $ are intimately connected with P\'olya urn models which are discussed, for example, in \cite{Berg} and \cite{KurzMN}.

In a sense, the parameter $z$ reflects the `polarisation' inside the society.

\item\label{ex:det} We end the presentation of examples with a rather pathological class, in fact one we are going to exclude below.
Suppose that for $\mu$-almost all $z$ either $\rho^{z}=\delta_{1}$ or $\rho^{z}=\delta_{-1}$. Then the popular vote is always unanimous. So, in a sense, there is little randomness in this example.
\end{enumerate}
\end{example}

\section{Democracy Deficit and Optimal Weights}\label{sec:demdef}

We want to choose the weights so that the democracy deficit is minimal.
By taking partial derivatives of $\Delta_{\sigma}$ with respect to each $w_\lambda$, we obtain a system of linear equations
that characterizes the optimal weights. Indeed, for $\lambda=1,\ldots, M$,
\begin{equation}
\sum_{\nu=1}^{M}\;\mathbb{E}\left(\chi_{\lambda}\chi_{\nu}\right)w_{\nu}~
=~\frac{1}{\sigma}\mathbb{E}\left(\chi_{\lambda}S\right)\,.\label{eq:LES}
\end{equation}
Defining the matrix $A$, the weight vector $w$ and the vector $b$ on the right hand side
of  \eqref{eq:LES} by
\begin{align}
A~&:=~\left(A_{\lambda\nu}\right)_{\lambda,\nu=1,\ldots,M}~
:=~\mathbb{E}\left(\chi_{\lambda}\chi_{\nu}\right)\label{eq:A}\\
w~&:=~\left(w_{\lambda}\right){}_{\lambda=1,\ldots,M}\notag\\
b~&:=~\left(b_{\lambda}\right)_{\lambda=1,\ldots,M}~
:=~\frac{1}{\sigma}\,\mathbb{E}\left(\chi_{\lambda}S\right)\,.\label{eq:b}
\end{align}
we may write \eqref{eq:LES} in matrix form as
\begin{align}\label{eq:LESM}
   A\;w~=~b\,.
\end{align}
Since the quantity $b$ depends on $\sigma $ (by a factor of $\frac{1}{\sigma} $),
the optimal weights do as well.

\medskip
A solution $w$ of \eqref{eq:LESM} is indeed a minimum if the matrix $A$, the Hessian of $\Delta $,  is (strictly) positive definite.

In this case, the matrix $A$ is invertible and consequently there is a unique tuple of optimal weights, namely the unique solution of \eqref{eq:LESM}.

If the groups vote independently of each other, the matrix $A$ is diagonal. This happens for
CBM($\mu,\rho $)-measures for which $\rho $ is independent of $z$. These cases are treated in
\cite{HOEC}.

It turns out that in the general case the matrix $A$ is indeed invertible under rather mild conditions.

\begin{defn}\label{def:suffrand}
   We say that a voting measure $\IP $ on $\prod_{\lambda=1}^{M}\,\{ -1,1 \}^{N_{\lambda}}$ is \emph{sufficiently random} if
\begin{align}\label{eq:suffrand}
   \IP\, (\chi_{1}=c_{1},\ldots,\chi_{M}=c_{M})~>~0\qquad \text{for all } c_{1},\ldots,c_{M}\in\{ -1,1 \}.
\end{align}
\end{defn}

Note that \eqref{eq:suffrand} is not very restrictive. For example, if the support $\rm{supp}\,\IP$ of the measure $\IP$ is the whole space $\{ -1,1 \}^{N}$, then $\IP$ satisfies \eqref{eq:suffrand}.
Moreover, for CBMs, we have:

\begin{prop}\label{prop:suffrand}
   Suppose that $\IP $ is a CBM($\mu,\rho$)-measure. Then $\IP$ is sufficiently random if and only if
\begin{align}\label{eq:noDelta}
   \mu\,\Big\{z\;\Big\vert \; \rho^{z}\{ -1 \}=1 \quad\text{\rm or}\quad \rho^{z}\{ 1 \}=1\Big\}~<~1.
\end{align}
\end{prop}

\begin{rem}
   If $\mu\{z\mid\rho^{z}\{ -1 \}=1 \;\text{\rm or}\; \rho^{z}\{ 1 \}=1\}=1 $, then the voting result in each
group is unanimous, so weights proportional to $N_{\lambda}$ are optimal weights (not necessarily
unique).
\end{rem}

\begin{prop}\label{prop:posdef} Let $\IP $ be a voting measure and let $A $ be defined by \eqref{eq:A}.
\begin{enumerate}
\item The  matrix $A$ is positive semi-definite.
\item $A$ is positive definite if $\IP$ is sufficiently random.
\end{enumerate}
\end{prop}

\begin{proof}
 For vectors $x,y\in\IR^M$, we will write $(x,y)$ for the Euclidean inner product.  For any vector $x=(x_{1},\ldots,x_{M}) $, we have
\begin{align}\label{eq:psd}
   \big(x,A\,x\big)~=~\IE\left(\left(\sum_{\lambda=1}^{M} x_{\lambda}\,\chi_{\lambda}\right)^{2} \right)~\geq~0\,.
\end{align}
So $A $ is positive semi-definite.

Suppose now that $(x,Ax)=0$. Then
\begin{align*}
   \IE\left(\left(\sum_{\lambda=1}^{M} x_{\lambda}\,\chi_{\lambda}\right)^{2} \right)~=~0\,.
\end{align*}
This implies that
\begin{align}\label{eq:eigen}
   \sum_{\lambda=1}^{M} x_{\lambda}\,\chi_{\lambda}~=~0 \qquad{\text{almost surely.}}
\end{align}
For a sufficiently random model, this is only possible if $x=0$.
\end{proof}

\begin{thm}\label{thm:weights}
   If the voting measure $\IP $ is sufficiently random, the optimal weights minimising the democracy deficit $\Delta_{\sigma} $ are unique and given by
\begin{align}\label{eq:w}
   w~=~A^{-1}\,b.
\end{align}
\end{thm}

\begin{defn}\label{def:wb}
  If $w$ satisfies \eqref{eq:w}, we set
\begin{align*}
   \overline{w}_{\nu}~:=~\frac{w_{\nu}}{\sum_{\lambda=1}^{M} w_{\lambda}},
\end{align*}
and call $\overline{w}_{\nu} $ the \emph{normalised optimal weights}.
\end{defn}

While the weights $w$ depend on $\sigma $ through $b=b_{\sigma}$, the normalised weights $\overline{w}$
are independent of $\sigma $. The $\overline{w}_{\nu} $ sum up to $1 $.

For the rest of this paper, we shall always assume that our models are sufficiently random.

Given Theorem \ref{thm:weights}, one is tempted to believe that the problem of optimal weights
is solved. Unfortunately, this is not the case, because it is practically impossible to compute the
ingredients like $\IE(\chi_{\lambda}\chi_{\nu})$ and $\IE(S\chi_{\lambda}) $ for finite (but fairly large) $N$.
A way out is to compute these quantities approximately for $N\to\infty$, and this is what we are doing throughout the rest of this paper.

\section{Asymptotics for the Collective Bias Model}\label{sec:asymp}
For given $\mu $ and $\rho^{z} $ and for $\uN=(N_{1},\ldots,N_{M})$, we denote by
$\IP_{\uN} $ the CBM($\mu,\rho$)-measure on $\prod_{\lambda=1}^{M}\{ -1,1 \}^{N_{\lambda}}$.
In the following, we try to compute optimal weights for large $N=\sum_\lambda N_{\lambda}$. More precisely, we consider
\eqref{eq:LESM} for $N\to\infty$. This limit is always taken in the sense that
\begin{align}\label{eq:Nlambda}
   \lim_{N\to\infty} \frac{N_{\lambda}}{N}~=~\alpha_{\lambda}~>~0
\end{align}
for each $\lambda$, and we assume that each group's population diverges to infinity as $N$ goes to infinity. Observe that $\sum_\lambda \alpha_{\lambda}=1 $. The constants $\alpha_\lambda$ represent the population of each group as a fraction of the overall population (at least asymptotically). Whenever the $N_{\lambda} $ are clear from the context we write $\IP_{N}, \IE_{N}$ instead of $\IP_{\uN}, \IE_{\uN}$, etc. We also set
\begin{align}
\big(A_{N}\big)_{\lambda\nu}~&:=~\IE_{N}\big(\chi_{\lambda}\chi_{\nu}\big),\quad \big(b_{N}\big)_{\lambda}~:=~\IE_{N}\left(\frac{S}{N}\chi_{\lambda}\right), \notag\\\quad\text{and}\quad s_{N}~&:=~\IE_{N}\left(\left(\frac{S}{N}\right)^{2}\right). \label{eq:AN}
\end{align}
Then
\begin{align}
   \Delta_{N}(w)~=~s_{N}\;-2\,(w,\, b_{N})\;+(w,A_{N}\,w). \label{eq:DeltaN}
\end{align}
In the above formulas, we set $\sigma:=N$.
From now on, we assume that $\IP$ is sufficiently random, i.e.\! that \eqref{eq:noDelta} holds.
Moreover, to avoid discussing different cases we also assume that $\rho $ is not trivial in the sense that
\begin{align}\label{eq:rhotriv}
   \mu\{z\mid \rho^{z}=\delta_{0}  \}~<~1\,.
\end{align}
If \eqref{eq:rhotriv} is violated, all voters act independently of each other. This is the `Impartial
Culture' and Penrose's square root law holds (see e.g.$\!\,$ \cite{FelsenthalM} or \cite{HOEC}).

The following result is the key observation which allows us to
evaluate important quantities asymptotically. This theorem explains the large population behaviour of a CBM.
\begin{thm}\label{thm:convthm}
   Suppose that the functions $f_{\lambda}:[-1,1]\to\IR, \lambda=1,\ldots,M$, are continuous on $[-1,0)\cup(0,1] $, and assume that the limits $f_{\lambda}(0+)=\lim_{t\searrow 0}f_{\lambda}(\alpha_{\lambda}t)$ and
$f_{\lambda}(0-)=\lim_{t\nearrow 0}f_{\lambda}(\alpha_{\lambda}t)$ exist. Set
\begin{align*}
   I_{z}\big(f_{\lambda}\big)~:=~\int_{[-1,0)\cup(0,1] } f_{\lambda}(\alpha_{\nu}\,t)\;\rho^{z}(\textup{d}t)\,+\,\frac{1}{2}\big(f_\lambda(0+)+f_\lambda(0-)\big)\; \rho^{z}\{ 0 \}.
\end{align*}

Then
\begin{align}\label{eq:convthm}
   &\IE\,\left( f_{1}\left(\frac{1}{N}S_{1}\right)\,\cdot\,\ldots\,\cdot f_{M}\left(\frac{1}{N}S_{M}\right)\right)~\rightarrow~
\int\,  I_{z}\big(f_{1}\big)\,\cdots\, I_{z}\big(f_{M}\big)
\; \mu(\textup{d}z).
\end{align}

\end{thm}
We could handle functions $f_{\lambda}$ with discontinuities (and left and right limits) in other points than $0$ as well, but we need the result only in the above form. The proof below, however, works for the more general case as well.
Theorem \ref{thm:convthm} says that the normalised voting margins $S_\lambda / N $ follow a distribution given by $\rho^z$ and $\mu$ in the large population limit. We can take transformations $f$ of these voting margins and their behaviour will be described by the distributions $\rho^z$ and $\mu$. Chief among these transformations will be the council vote $\chi_\lambda = \chi_\lambda\left(S_\lambda\right)$ cast by each group which presents a point of discontinuity at 0.
\begin{proof}
   By the strong law of large numbers, we get
\begin{align*}
   P_{t}\left(\lim_{n\to\infty} \frac{1}{N_{\lambda}}S_{\lambda}=t\right)~=~1.
\end{align*}
So, if $f$ is continuous on $[-1,1] $, it follows that
\begin{align*}
   \int E_{t}\left(f\left(\frac{1}{N}S_{\lambda}\right) \right)\,\rho^{z}(\textup{d}t)~\to~\int f(\alpha_{\lambda} t)\,\rho^{z}(\textup{d}t)
\end{align*}
for all $z$. From this, \eqref{eq:convthm} follows for continuous $f_{\lambda}$.
To prove \eqref{eq:convthm} in the general case, we observe that
for $N\to\infty$
\begin{align*}
  &P_{0}\left(\frac{1}{N_{\lambda}}S_{\lambda}>0\right)~\to~\frac{1}{2}
\quad\text{and}\quad P_{0}\left(\frac{1}{N_{\lambda}}S_{\lambda}<0\right)~\to~\frac{1}{2}.
\end{align*}
\end{proof}

\begin{defn}\label{def:notation}
   We introduce the following notation for further use:
\begin{align*}
   & m_{1}(\rho)~=~\int t \,\rho^{z}(\textup{d}t), && m_{2}(\rho)~=~\int t^2 \,\rho^{z}(\textup{d}t),\\
    & \om_{1}(\rho)~=~\int |t| \,\rho^{z}(\textup{d}t), && d(\rho)~=~\rho^{z} (0,1]-\rho^{z} [-1,0). \notag
\end{align*}
Note that the above quantities depend on $z $. These quantities are important characteristics of the measures $\mu$ and $\rho$. They measure the strength of the group bias for different values $z$ of the global bias. E.g.\! a positive $m_1(\rho)$ close to 1 implies that the group bias measure $\rho^z$ induces, on average, a strong bias in favour of the issue being considered. Whereas $\om_1(\rho)$ can be interpreted as a measure of intra-group cohesion, the product of $m_1$ and $d(\rho)$ is a measure of inter-group cohesion. These two measures will allow us to compare how strong the intra-group cohesion is versus the inter-group cohesion. These quantities will be used in the calculation of the optimal weights.

For any function $\varphi$ on $[-1,1] $, we  introduce the shorthand notation
\begin{align*}
   \langle \,\varphi\, \rangle~=~\int \varphi(z)\,\mu(\textup{d}z).
\end{align*}
\end{defn}

\begin{thm}\label{thm:limits}
   Assume \eqref{eq:noDelta}, \eqref{eq:Nlambda} and \eqref{eq:rhotriv}. Then
\begin{align*}
   (A_{N})_{\lambda \nu}~&\to~ a~:=~\langle d(\rho)^{2} \rangle,\;\lambda\neq\nu,\\
    (b_{N})_{\lambda}~&\to ~b_{\lambda}~:=~ \langle\om_{1}(\rho)-m_{1}(\rho)d(\rho)\rangle\,\alpha_{\lambda}\;+\;\langle m_{1}(\rho)\,d(\rho) \rangle,\\
\textrm{s}_{N}~&\to~\textrm{s}~:=
~\sum_{\nu=1}^{M}\,\alpha_{\nu}^{2} \big(\langle m_{2}(\rho)\rangle-\langle {m_{1}(\rho)}^{2} \rangle\big) + \langle m_{1}(\rho)^{2} \rangle.
\end{align*}
\end{thm}
Theorem \ref{thm:limits} follows immediately from Theorem \ref{thm:convthm}.

Informally speaking, Theorem \ref{thm:limits} says that the minimisation problem \eqref{eq:DeltaN}
`converges' to the minimisation problem
\begin{align}\label{eq:DeltaInf}
   \min\;\Delta_{\infty}(v_{1},\ldots,v_{M})~=~s\;-2\,(v,\, b)\;+(v,A\,v).
\end{align}
In the following, we try to explore the validity of this informal idea. The following theorem implies that for positive definite limiting coefficient matrices $A$ the optimal weights of the finite population problem converge to the optimal weights of the asymptotic problem.
\begin{thm}\label{cor:A}
   The matrices $A_{N}$ converge (in operator norm) to the matrix
\begin{align}\label{eq:Ainf}
   A_{\lambda\nu}~=~\left\{
                      \begin{array}{ll}
                        1, & \hbox{if $\lambda=\nu $,} \\
                        a, & \hbox{otherwise,}
                      \end{array}
                    \right.
\end{align}
with $a=\langle d(\rho)^{2} \rangle$.

Moreover, $A$ is positive semi-definite. $A$ is positive definite if $a<1$. In this case,
\begin{align}\label{eq:konvinv}
  {A_{N}}^{-1}~\to~A^{-1}
\end{align}
and
\begin{align}\label{eq:Ainv}
   {\left(A^{-1}\right)}_{\lambda\nu}~=~\frac{1}{D}\;\left\{
                           \begin{array}{ll}
                             1+(M-2)a, & \hbox{if $\lambda=\nu$,} \\
                             -a, & \hbox{otherwise,}
                           \end{array}
                         \right.
\end{align}
 where $D=(1-a)\big((1+(M-1)a\big)$.
\end{thm}

\begin{proof}
We note that $0\leq a\leq 1$.
Since, for any $x\in\IR^{M}$,
\begin{align*}
   \left(x,Ax\right)~=~(1-a)\,\sum_{\lambda=1}^{M}x_{\lambda}^{2}\;+\;a\,\left(\sum_{\lambda=1}^{M}x_{\lambda}\right)^{2},
\end{align*}
we see that $A$ is positive semi-definite in general and positive definite if $a<1$.

Let $I$ stand for the $M\times M$ identity matrix. To prove \eqref{eq:konvinv} we compute,
\begin{align}
   \|{A_{N}}^{-1}-A^{-1}\|~&=~\Big\|\,A^{-1}\,\Big(\big(I+(A_{N}-A)A^{-1}\big)^{-1}-I\Big)\,\Big\|\notag\\
&\leq~\|A^{-1}\|\;\sum_{k=1}^{\infty}\,\|A_{N}-A\|^{k}\,\|A^{-1}\|^{k}\notag\\
&=~\frac{\|A^{-1}\|^2\,\|A_{N}-A\|}{1-\|A^{-1}\|\|A_{N}-A\|}\label{eq:exp}.
\end{align}
Since $\|A_{N}-A\|$ tends to 0, \eqref{eq:exp} goes to 0 as well.
The claim \eqref{eq:Ainv} follows by direct
calculation.
\end{proof}
\begin{defn}
   We say that the collective bias model CBM$(\mu,\rho)$ is \emph{tightly correlated} if $a=\langle d(\rho)^{2} \rangle=1$.
\end{defn}
As we will see, tight correlation implies that all groups end up voting unanimously in the council. For now, we characterise tight correlation in terms of the probabilities assigned by $\rho$ for different values $z$ of the global bias. The key idea is that $\rho^z$ assigns probability 1 to either $(0,1]$ or $[0,1)$ for ($\mu$-almost) all $z$, and thus all group biases will be of the same sign, inducing the aforementioned unanimous council vote.
\begin{prop}\label{prop:tight}
   The collective bias model $CBM(\mu,\rho)$ is tightly correlated if and only if
for $\mu$-almost all $z$ either $\rho^{z} (0,1]=1$ or $\rho^{z} [-1,0)=1$ holds.
\end{prop}
\begin{proof}
   Since $0\leq d(\rho)^{2}\leq 1$ for all $z$, we have $\xi:=1-d(\rho)^{2}\geq 0$ and $\int \xi\,d\mu=0$ implies
    $\xi=0$ $\mu$-almost surely. It follows that $|d(\rho)|=1$ for $\mu$-almost all $z $, so $\rho^{z} (0,1]=1$
or $\rho^{z} [-1,0)=1$.
\end{proof}

\section{Optimal Weights}\label{sec:optw}
In this section, we investigate the asymptotics of the optimal weights of CBMs
for large $N$. As above, we assume \eqref{eq:noDelta}, \eqref{eq:Nlambda}, and \eqref{eq:rhotriv} for the rest of this paper.

The tightly correlated case needs a different treatment, so we first assume that the model CBM$(\mu,\rho)$ is \emph{not} tightly correlated, i.e.$\!\,$ that $a=\langle d(\rho)^{2} \rangle<1$, in this section. Section \ref{sec:optwtc} discusses the tightly correlated case.

By Theorem \ref{thm:weights}, for fixed $N $, there are unique optimal weights $w_{N} $.
\begin{thm}\label{thm:asweights}
   If the model $CBM(\mu,\rho) $ is not tightly correlated, then the optimal weights $w^{(N)}$, i.e.\!
the minima of $\Delta_{N} $, converge
for $N\to\infty$ to the minima of $\Delta_{\infty}$ (defined in \eqref{eq:DeltaInf}), and these weights $w_{\lambda}$ are given by
\begin{align}
   w_{\lambda}~=~C_{1}\,\alpha_{\lambda}\;+\;C_{2}\label{eq:optw},
\end{align}
with coefficients depending on $\mu,\rho$, and $M$ but not on the $\alpha_{\lambda}$.

More precisely,
\begin{align}
   C_{1}~&=~\frac{1}{1-a}\,\Big(\langle \om_{1}(\rho) \rangle-\langle m_{1}(\rho)\, d(\rho) \label{eq:C1} \rangle\Big)\\
\text{and}\quad C_{2}~&=~\frac{1}{1-a}\,\frac{\langle m_{1}(\rho)\,d(\rho) \rangle-a\langle \om_{1}(\rho)\rangle}{1+(M-1)\,a}.\label{eq:C2}
\end{align}

Moreover,
\begin{align}\label{eq:sumw}
   \sum_\lambda w_{\lambda}~=~
\frac{\langle \om_{1}(\rho) \rangle\,+\,(M-1)\langle m_{1}(\rho)d(\rho) \rangle}{1+(M-1)\,a}.
\end{align}
\end{thm}
Theorem \ref{thm:asweights} follows from Theorems \ref{thm:limits} and \ref{cor:A} by a straightforward computation.
\begin{cor}Under the assumptions of Theorem \ref{thm:asweights},
the normalised weights $\overline{w}^{(N)} $ converge to
\begin{align}
   \overline{w}_{\lambda}~=~\overline{C}_{1}\,\alpha_{\lambda}\;+\;\overline{C}_{2}\label{eq:normw}
\end{align}
\end{cor}

\begin{rem}
   \begin{enumerate}
      \item By Theorem \ref{thm:asweights}, the optimal weights are always the sum of a term
            proportional to the size of the population and a term independent of the population.
            The weights of the states in the Electoral College of the U.S. constitution are
            precisely chosen in this fashion.
\item In the limit $a\to 0$, meaning that the groups are almost independent, the constant term in
\eqref{eq:optw} tends to 0, so that $\overline{w}_{\lambda}\to \alpha_{\lambda}$ which
is the result for the simple CBM (see \cite{HOEC}).
\item The sum of the weights \eqref{eq:sumw} is strictly positive and finite, even in the limit
    $a\to 1$. This indicates that the choice $\sigma=N$ is reasonable. In fact,
\begin{align}\label{eq:sumw1}
   \lim_{a\to 1} \sum_\lambda w_{\lambda}~=~\langle \om_{1}(\rho) \rangle\,.
\end{align}
   \end{enumerate}
\end{rem}
\begin{cor}
   Under the assumptions of Theorem \ref{thm:asweights}, the minimal democracy deficit $\Delta_{N} $ is asymptotically of the form
\begin{align*}
   \Delta_{\infty}~=~D_{1}\,\sum_{\lambda=1}^{M}{\alpha_{\lambda}}^{2}\;+\;D_{2}.
\end{align*}
\end{cor}
\begin{rem}
 The constants $D_{1}$ and $D_{2}$ depend on $\mu, \rho, M$ and can be computed from \eqref{eq:Ainv}, \eqref{eq:C1}, and \eqref{eq:C2}.
\end{rem}

\section{Optimal Weights for Tight Correlations}\label{sec:optwtc}

Now we turn to the case of tightly correlated models, i.e.$\!\,$ $a=1$.

Then, in the limit $N\to\infty$, setting $\sigma:=N $, equation \eqref{eq:LES} which describes the critical
points of $\Delta_{N} $ tends to $\tilde{A}w=b$ with
\begin{align*}
   \tilde{A}_{\lambda\nu}~=~1 \qquad\text{for all } \lambda,\nu\,.
\end{align*}
The matrix $\tilde{A}$ is degenerate. It has an $(M-1)$-fold degenerate eigenvalue at $0$ and a simple
eigenvalue at $M $.

The democracy deficit $\Delta_{N}$ tends to
\begin{align}
   \Delta_{\infty}~&=~\sum_{\lambda=1}^{M}\,\alpha_{\lambda}^{2}\,\big(\langle m_{2}(\rho) \rangle-\langle m_{1}(\rho)^{2} \rangle\big)\;\;+\langle m_{1}(\rho)^{2} \rangle\notag\\
&\quad-\;2\,\langle \om_{1}(\rho) \rangle\,\sum_{\lambda=1}^{M}\,w_{\lambda}\;+\left(\sum_{\lambda=1}^{M}\,w_{\lambda}\right)^{2}.\label{eq:da1}
\end{align}
\eqref{eq:da1} is an equation in $\sum_{\lambda}w_{\lambda}$. The extrema of $\Delta_{\infty}$ are all weights $w_{\lambda}$ such that
\begin{align*}
   \sum_{\lambda=1}^{M}\,w_{\lambda}~=~\langle \om_{1}(\rho) \rangle\,.
\end{align*}
This condition is in agreement with \eqref{eq:sumw1}.

\begin{thm}\label{thm:optwtight}
   Suppose $a=1$. If
\begin{align*}\sum_{\lambda=1}^{M} w_{\lambda}~ =~ \sum_{\lambda=1}^{M} v_{\lambda},
\end{align*}
then
\begin{align*}
   \Delta_{N}(w)~-\Delta_{N}(v)~\to~ 0\qquad\text{as } N\to\infty.
\end{align*}
In particular, any tuple $w$ of weights with $\sum_\lambda w_{\lambda}=\langle \om_{1}(\rho) \rangle$
is close to the minimal democracy deficit in the sense that
\begin{align*}
\Delta_{N}(w)~\to~\min_{v}\; \Delta_{\infty}(v).
\end{align*}
\end{thm}
Theorem \ref{thm:optwtight} implies that for large systems with tight correlation `it doesn't matter'
how the weights are distributed among the groups. This assertion is confirmed by the following observation:
\begin{thm}\label{thm:unani}
   If the model CBM($\mu,\rho $) is tightly correlated, then
\begin{align*}
   \IP\,\Big(S_{\lambda}>0  \text{ for all } \lambda \quad\text{or }\quad S_{\lambda}<0  \text{ for all } \lambda \Big)~\to~1\quad\text{as }N\to\infty.
\end{align*}
\end{thm}
Thus, in large tightly correlated systems, council votes are almost always unanimous! Consequently,
for $N\to\infty$,
any $w$ with $\sum_\lambda w_{\lambda}>0$ induces the same voting result in the council. This might be surprising at first. As the overall population goes to infinity, the probability of a unanimous council vote goes to 1. Hence, the limit of the optimality condition \eqref{eq:LES} is a linear equation system with an infinity of solutions. More precisely, any set of weights $w_1,\ldots,w_M$ that sum to a fixed positive value given by the limit of \eqref{eq:b} solves \eqref{eq:LES}. As such, the assignation of the voting weights only serves the purpose of appropriately scaling the magnitude of the (unanimous) council vote to bring it in line with $S/\sigma$. However, the constraint on the sum is not binding, as we know that any transformation of a weighted voting system that multiplies all weights by a positive constant while leaving the relative quota untouched is equivalent to the original voting system. Thus, a set of weights which sum to 1 is but a representative of an equivalence class of voting systems. The selection of the optimal weights when a unanimous council vote occurs with high probability is a trivial problem.
\begin{proof}
Set
\begin{align*}
Z_{+}~=~\{ z\in[-1,1]\mid  \rho^{z} (0,1]=1 \} \quad \text{and}\quad
Z_{-}~=~\{ z\in[-1,1]\mid  \rho^{z} [-1,0)=1\}\,.
\end{align*}
Since the measure $\IP $ is tightly correlated, we have due to Proposition \ref{prop:tight} that
\begin{align*}
   Z_{+}\;\cup\; Z_{-}~=[-1,1]\qquad\text{up to a set of $\mu $-measure $0$}\,.
\end{align*}
In particular, $\rho^{z}\not=\delta_{0} $ for $\mu $-almost all $z $, so $\IP(S_{\lambda}=0)\to 0 $
for any $\lambda $.

Thus, it suffices to prove that for any given $\nu\not=\lambda $
\begin{align*}
   \IP(S_{\nu}>0, S_{\lambda}<0)\to 0.
\end{align*}

For $t\in (0,1]$, we have
\begin{align*}
P_{t}(S_{\lambda}<0)~\to~ 0;
\end{align*}
 thus, for $z\in Z_{+}$,
\begin{align*}
\int\;P_{t}(S_{\lambda}<0)\;\rho^{z}(\textup{d}t)~\to~ 0
\end{align*}
and similarly, for $z\in Z_{-}$,
\begin{align*}
\int\;P_{t}(S_{\nu}>0)\;\rho^{z}(\textup{d}t)~\to~ 0.
\end{align*}
Hence
\begin{align*}
   \IP(S_{\nu}>0, S_{\lambda}<0)~\leq~&\int_{Z_{+}}\,\int P_{t}(S_{\lambda}<0)\;\rho^{z}(\textup{d}t)\;\mu(\textup{d}z)
~+~\int_{Z_{-}}\,\int P_{t}(S_{\nu}>0)\;\rho^{z}(\textup{d}t)\;\mu(\textup{d}z) \to~ 0.
\end{align*}
\end{proof}

\section{Specific Models}\label{sec:specmod}
In this section, we analyse some models from Example \ref{ex:1}. In these
examples, we can compute relevant quantities explicitly.
\subsection{Additive Models}\label{sec:additive}
We start with some additive models as in Example \ref{ex:1}.\ref{ex:add} with specific bias measures $ \mu$ and $\rho $.

We recall that for additive models the voting measure $\IP\big(\ux_{1},\ux_{2},\ldots,\ux_{M}\big)$ is given by
\begin{align}\label{eq:addcbm3}
\int  \left( \int P_{z+y_{1}}(\ux_{1})\,\rho(\textup{d}y_{1})
\cdots\int P_{z+y_{M}}(\ux_{M})
\,\rho(\textup{d}y_{M}) \right)\;\mu(\textup{d}z).
\end{align}
$\IP$ is indeed a voting measure if both $\mu$ and $\rho $ are symmetric, i.e.\! $\mu [a,b]=\mu [-b,-a] $ and similarly for $\rho $. $\IP$ is sufficiently random except for the (pathological) case $\mu=\frac{1}{2}\big(\delta_{1}+\delta_{-1}\big)$ and $\rho=\delta_{0} $. $\IP $ is tightly correlated if (and only if) for $\mu$-almost all $z$ either $\rho (-z,1]=1$ or $\rho [-1,-z)=1$.

\subsubsection{Uniform Distribution with Weak Global Bias}
In our first example, we take $\mu $ and $\rho $ to be the uniform probability distribution on $[-g,g] $ (for `global' bias) and
$[-\ell,\ell] $ (`local' bias), respectively. We assume first that $g\leq \ell$, indicating that the (average) global
bias is not bigger than the (average) local bias. So the voting measure $\IP(\underline{x}_{1},\ldots,\underline{x}_{M})$ is given by
\begin{align}\label{eq:addcbm}
   \frac{1}{2g}\int_{-g}^{+g}\;\left(\,\frac{1}{2\ell}\int_{z-\ell}^{z+\ell} P_{t_{1}}(\underline{x}_{1})\,\textup{d}t_{1}\cdots
\frac{1}{2\ell}\int_{z-\ell}^{z+\ell} P_{t_{M}}(\underline{x}_{M})\,\textup{d}t_{M}\,\right)\,\textup{d}z.
\end{align}
For this specific example, we can explicitly compute the relevant quantities from Definition \ref{def:notation} and Theorem \ref{thm:asweights}. By a straightforward but tedious computation, we obtain:
\begin{align}
   &a~=~\frac{1}{3}\,\frac{g^{2}}{\ell^{2}}~\leq~\frac{1}{3},&& \langle m_{1}(\rho)\,d(\rho)
\rangle~=~\frac{1}{3}\,\frac{g^{2}}{\ell},\notag\\
&\langle \om_{1}(\rho)
\rangle~=~\frac{1}{6\ell}\;\big(3\ell^{2}+\,g^{2}\big),
&&\langle \om_{1}(\rho)\rangle- \langle m_{1}(\rho)\,d(\rho)\rangle
~=~\frac{1}{6\ell}\big(3\ell^{2}-g^{2}\big),\notag\\
&\langle m_{1}(\rho)d(\rho)\rangle - a \langle \om_{1}(\rho)\rangle~=~ \frac{g^{2}}{18\ell^{3}}\big(3\ell^{2}-g^{2}\big). \label{eq:addcbm-1}
\end{align}
This gives
\begin{thm}\label{opt_weights_uniform}
   For the additive CBM in \eqref{eq:addcbm} with $g\leq\ell$, the optimal weights are
\begin{align}\label{eq:waddgkl}
   w_{\lambda}~=~\frac{1}{2}\ell\,\alpha_{\lambda}\;+\;\frac{1}{2}\,\frac{g^{2}\ell}{3\ell^{2}+(M-1)g^{2}}.
\end{align}
\end{thm}
\begin{rem}
   \begin{enumerate}
      \item If there is no global bias (meaning $g\searrow 0$), we obtain the result for independent groups, i.e.$\!\,$ the weights are proportional to $\alpha_{\lambda}$.
\item The quantity $\langle m_{1}d \rangle-a\langle \om_{1} \rangle$ is non-negative. This is not always the case as we will see in Section \ref{sec:nonneg}.
   \end{enumerate}
\end{rem}
\subsubsection{Uniform Distribution with Strong Global Bias}
Now, we turn to the case $\ell\leq g$. In this case, we compute:
\begin{align}
   &a~=~1-\frac{2}{3}\,\frac{\ell}{g}~<~1,&& \langle m_{1}(\rho)\,d(\rho)
\rangle~=~\frac{1}{6g}\,\big(3g^{2}-\ell^{2}\big),\notag\\
&\langle \om_{1}(\rho)
\rangle~=~\frac{1}{6g}\;\big(3g^{2}+\,\ell^{2}\big),
&&\langle \om_{1}(\rho)\rangle- \langle m_{1}(\rho)\,d(\rho)\rangle
~=~\frac{1}{3}\frac{\ell^{2}}{g},\notag\\
&\langle m_{1}(\rho)d(\rho)\rangle - a \langle \om_{1}(\rho)\rangle~=~ \frac{\ell}{9g^{2}}\big(3g^{2}-3g\ell +\ell^{2}\big). \label{eq:addcbm-2}
\end{align}

\begin{thm}
   For the additive CBM in \eqref{eq:addcbm} with $\ell\leq g$, the optimal weights are
\begin{align}\label{eq:waddlkg}
   w_{\lambda}~=~\frac{1}{2}\ell\,\alpha_{\lambda}\;+\;\frac{1}{2}\,\frac{3 g^{2}-3g\ell +\ell^{2}}{3Mg -2(M-1)\ell}.
\end{align}
\end{thm}
\begin{rem}
   \begin{enumerate}
      \item For the case $g=\ell$, formulae \eqref{eq:waddgkl} and \eqref{eq:waddlkg} agree.
\item In the limit $\ell\to 0 $ we find $a\to 1 $, i.e.$\!\,$ we approach the tightly correlated case. In this case, the weights become constant, independent of the sizes of the groups.
This limit case corresponds to Impartial Anonymous Culture for the \emph{union}.
   \end{enumerate}
\end{rem}

\subsubsection{Global Bias Concentrated in Two Points}
We study an additive model for which the global bias may assume the value $+g,-g$ with probability $\frac{1}{2}$ each. The local bias is uniformly distributed on $[-\ell,\ell] $ with $0<g<\ell$.

We just give the final result: the optimal weights are
\begin{align*}
    w_{\lambda}~=~\frac{1}{2}\ell\,\alpha_{\lambda}\;+\;\frac{1}{2}\,\frac{\ell g^{2}}{(M-1)g^{2} +\ell^{2}}.
\end{align*}
In this setting, the tightly correlated case is approached in the limit $\ell\searrow g $.
The weights tend to $w_{\lambda}\to \frac{1}{2}g\,\alpha_{\lambda}+ \frac{1}{2}\frac{g}{M}$ as $\ell\searrow g $.

\subsection{Multiplicative Models}\label{sec:mult}
We now analyse the multiplicative models in Example \ref{ex:1}.\ref{ex:mult}. If $\mu\{0\}=0$, the model is tightly correlated if and only if $\supp\,\rho\subset(0,1]$ or $\supp\,\rho\subset[-1,0)$.

Again, we consider uniform distributions on $[\ell_1,\ell_2]$ and $[-g,g]$, respectively, in more detail. The probability of each configuration $\ux_{\lambda}\in\{ -1,1 \}^{N_{\lambda}},\lambda=1,\ldots,M$, is
\begin{align}\label{eq:multcbm}
   \frac{1}{2g}\int_{-g}^{+g}\;
\left(\,\frac{1}{\ell_{2}-\ell_{1}}\int_{\ell_{1}}^{\ell_{2}} P_{z y_{1}}(\underline{x}_{1})\,\textup{d}y_{1}\cdots
\frac{1}{\ell_{2}-\ell_{1}}\int_{\ell_{1}}^{\ell_{2}} P_{z y_{M}}(\underline{x}_{M})\,\textup{d}y_{M}\,\right)\,\textup{d}z\,.
\end{align}
So while the global bias measure is uniform on the interval $[-g,g]$, the local bias modifier $\rho$ is uniform on the interval $[\ell_1,\ell_2]$. If $\ell_{1}\geq 0$ (or $\ell_{2}\leq 0 $, which gives the same model class), the model is tightly correlated.

Assuming $\ell_{1}<0<\ell_{2} $, we obtain
\begin{align*}
   &a~=~\frac{(\ell_{2}+\ell_{1})^{2}}{(\ell_{2}-\ell_{1})^{2}}~<~1, & \langle m_{1}(\rho)\,d(\rho)
\rangle~=~\frac{g}{4}\,\frac{(\ell_{2}+\ell_{1})^{2}}{(\ell_{2}-\ell_{1})},\notag\\
&\langle \om_{1}(\rho)
\rangle~=~\frac{g}{4}\,\frac{\ell_{2}^{2}+\ell_{1}^{2}}{(\ell_{2}-\ell_{1})},
&\langle \om_{1}(\rho)\rangle- \langle m_{1}(\rho)\,d(\rho)\rangle
~=~-\,\frac{g}{2}\,\frac{\ell_{1}\ell_{2}}{\ell_{2}-\ell_{1}},\notag\\
&\langle m_{1}(\rho)d(\rho)\rangle - a \langle \om_{1}(\rho)\rangle~=~
-\,\frac{g}{2}\,\ell_{1}\ell_{2}\,\frac{(\ell_{2}+\ell_{1})^{2}}{(\ell_{2}-\ell_{1})^{3}}~\geq~0.
\end{align*}
So, for the optimal weights according to Theorem \ref{thm:asweights}, we obtain
\begin{align}
   w_{\lambda}~&=~\frac{g}{8}\,(\ell_{2}-\ell_{1})\,\alpha_{\lambda}\;
+\;\frac{g}{8}\,(\ell_{2}-\ell_{1})\,\frac{(\ell_{2}+\ell_{1})^{2}}{(\ell_{2}-\ell_{1})^{2}+
(M-1)(\ell_{2}+\ell_{1})^{2}}, \label{eq:mult1}\\
\intertext{or, equivalently,}\tilde{w}_{\lambda}~&=~\alpha_{\lambda}\;
+\;\frac{(\ell_{2}+\ell_{1})^{2}}{(\ell_{2}-\ell_{1})^{2}+
(M-1)(\ell_{2}+\ell_{1})^{2}}\,.\label{eq:mult2}
\end{align}
For $\ell_{1}\nearrow 0$ approaching the tightly correlated case, we get
\begin{align*}
   \tilde{w}_{\lambda}~&=~\alpha_{\lambda}\;
+\;\frac{1}{M}\,.
\end{align*}
Moreover, we observe that the formulae \eqref{eq:mult1} and \eqref{eq:mult2} make sense even in the
tightly correlated case, i.e.$\!\,$ for $\ell_{1}\geq 0$.

Next we turn to the case where $\rho(0,1]=1$ while maintaining the condition \eqref{eq:rhotriv}.  Then there are only two possibilities: either the model is tightly correlated and the optimal weights are indeterminate. This is the case if and only if $\mu\{0\}=0$. The complementary case is $0<\mu\{0\}<1$. We can interpret this as the existence of some fraction of the issues which are not subject to any global bias. The multiplicative structure of the local bias means all voters make up their own minds on these issues. We can determine the optimal weights in this case without placing any additional assumptions on the bias measures $\mu$ and $\rho$.

The key observation is that for $\mu$-almost all $z$ the equality $m_{1}(\rho)\,d(\rho) = \om_{1}(\rho)$ holds. The model is not tightly correlated, nor are the voters belonging to different groups independent. So we have $0<a<1$ and
\begin{align*}
   w_{\lambda}~&=~\frac{\langle \om_{1}(\rho) \rangle}{1+(M-1)\,a}\quad
\text{or, equivalently,}\quad\overline{w}_{\lambda}~=~\frac{1}{M}\,.
\end{align*}
In conclusion, for this model, the optimal weights have to be chosen equal for all groups $\lambda$, no matter their size $\alpha_\lambda$. As for the intuition behind this result, let us recall the formula for optimal weights in non-tightly correlated models given in Theorem \ref{thm:asweights}: the optimal weights are given by $C_1 \alpha_\lambda+C_2$, with
\begin{align*}
   C_{1}~&=~\frac{1}{1-a}\,\Big(\langle \om_{1}(\rho) \rangle-\langle m_{1}(\rho)\, d(\rho) \rangle\Big)\\
\text{and}\quad C_{2}~&=~\frac{1}{1-a}\,\frac{\langle m_{1}(\rho)\,d(\rho) \rangle-a\langle \om_{1}(\rho)\rangle}{1+(M-1)\,a}.
\end{align*}
As mentioned after Definition \ref{def:notation}, $\om_{1}(\rho)$ can be interpreted as a measure of intra-group cohesion, and $m_{1}(\rho)\,d(\rho)$ as a measure of inter-group cohesion. The equality of these two in the present example is intuitively due to the fact that $Z=0$ induces zero cohesion both within each group as well as across group boundaries, and for $Z\neq0$ we have a very strong correlation of all voters due to the assumption $\rho(0,1]=1$ which implies that the sign of each group bias will always be the same as the sign of the global bias. Hence, we have $\langle m_{1}(\rho)\,d(\rho) \rangle = \langle \om_{1}(\rho) \rangle$, and $C_1=0$ implies there is no proportional component to the optimal weights. The constant component $C_2$, on the other hand, does not disappear, because the fraction of issues for which there is independent voting (i.e.\! those for which $Z=0$), induces a non-tight correlation and $a<1$.
This example sheds some light on where the summands in the optimal weight formula in Theorem \ref{thm:asweights} come from: as mentioned previously, the constant component $C_2$ is induced by the correlation between votes belonging to different groups. Now we see that the proportional component is a manifestation of the stronger cohesion within each group when compared to inter-group cohesion.

\section{Non-Negativity of the Weights}\label{sec:nonneg}
In applications on public voting procedures, negative weights would be rather absurd: the consent of
such a voter could decrease the majority margin or even change an `aye' to a `nay'. It seems likely that no group would accept being assigned a negative voting weight. Even if they did, this would not bring about a minimisation of the democracy deficit, since a group with negative weight would face incentives to misrepresent their true preferences. On the other hand, in an
estimation problem, i.e.$\!\,$ for estimating the magnitude of the voting margin, negative weights may make sense. It has been pointed out by an anonymous referee that negative weights may also make sense in an automated preference aggregation setting, when sincere voting can be assumed.

In Theorem \ref{thm:asweights}, we identified the optimal weights $w_{\lambda} $ as
\begin{align}\label{eq:weights2}
   w_{\lambda}~=~C_{1}\,\alpha_{\lambda}\;+\;C_{2}\,.
\end{align}
The constant $C_{1} $ is always non-negative. Moreover, in all explicit examples in Section \ref{sec:specmod} the constant $C_{2}$ turned out to be non-negative as well.

In general, the constant $C_{2}$ is non-negative if and only if
 \begin{align}\label{eq:nonneg}
    a\langle \om_1(\rho) \rangle \leq \langle m_{1}(\rho) \, d(\rho) \rangle \,.
 \end{align}
As it turns out, condition \eqref{eq:nonneg} \emph{can} be violated under certain assumptions on
the measures $\mu $ and $\rho^{z} $. Consequently, for small $\alpha_{\lambda} $, equation \eqref{eq:weights2} prescribes negative weights.

\subsection{An Example with Negative Optimal Weights}
To see that \eqref{eq:nonneg} can be violated, we consider an additive model with $\mu=\eh(\delta_{g}+\delta_{-g})$ and $\rho=\frac{1}{4}(\delta_{-\ell_{2}}+\delta_{-\ell_{1}}+\delta_{\ell_{1}}+\delta_{\ell_{2}}) $ and choose $0<\ell_{1}<g<\ell_{2} $ with $g+\ell_{2}\leq 1$.

Then, for the additive model with $\mu $ and $\rho $ we compute:
\begin{align*}
   a=\langle d(\rho)^{2} \rangle=\frac{1}{4},&&\langle m_{1}d(\rho) \rangle=\eh g,
   &&\langle \om_{1}(\rho)\rangle=\eh g + \eh \ell_{2}.
&\end{align*}

Consequently, the constant term $C_{2}$ in the optimal weight \eqref{eq:weights2} is \emph{negative} if $\ell_{2}>3 g$. In this case, the optimal weight is negative
for small $\alpha_{\lambda}$.

An analogous result holds for uniform distributions both for $\mu $ (around $\pm g$) and for $\rho $
around $\pm\ell_{1}$ and $\pm\ell_{2} $, as long as these six intervals are small enough.

In the remainder of this section, we will focus on the additive model and the problem of negative weights. For simplicity's sake, we will assume for the rest of Section \ref{sec:nonneg} that the support of both $\mu$ and $\rho$ belongs to $[-1/2,1/2]$.

\subsection{Non-Negativity of $w$ in Additive Collective Bias Models with $\mu=\rho$}

In this section, we consider the case where the central bias and the
group modifiers of an additive CBM follow the same distribution. Of course, all bias variables and modifiers
$Z$ and the $Y_{\lambda}$ are still assumed to be independent. So the random variables $Z,Y_1,\ldots,Y_M$ are all i.i.d. As noted in Section \ref{warm-up}, the assumption of identically distributed $Z$ and $Y_\lambda$ reflects that global bias and local bias each have the same influence on the voters, with neither of the two dominating. We will use the notation
\[
r:=\langle m_{1}(\rho)\,d(\rho) \rangle,\quad m:=\langle \om_{1}(\rho) \rangle.
\]
Recall that according to Theorem \ref{thm:asweights}, the optimal weights are proportional to
\[
w_{\lambda}=r-am+\left(1+\left(M-1\right)a\right)\left(m-r\right)\alpha_{\lambda}.
\]

We prove that for this setup the optimal weights can never be negative.
\begin{thm}
\label{thm:mu_rho}If $\mu=\rho$ in an additive CBM, the constant term in
the optimal weights $r-am$ is non-negative and $r-am=0$ holds if
and only if $\mu=\delta_{0}$. Furthermore, $0\leq a\leq1/3$, where
$a=0$ holds if and only if $\mu=\delta_{0}$, and $a=1/3$ if and
only if $\mu$ has no atoms, i.e., for all $x\in\mathbb{R}$, $\mu\left\{ x\right\} =0$.
\end{thm}

This theorem says -- among other things -- that the constant term
in the optimal weights $r-am$ is 0 if and only if $\mu=\rho=\delta_{0}$.
But the latter equality implies that all voters are independent, a
case which we discarded earlier. (Note that if $\mu=\rho=\delta_{0}$,
the optimal weights are \emph{not} proportional to the group sizes.
Instead, the square root law holds and the optimal weights are proportional
to $\sqrt{\alpha_{\lambda}}$.) Hence, by Theorem \ref{thm:mu_rho},
for all $\mu=\rho\neq\delta_{0}$, the optimal weights are the sum
of a positive constant $r-am>0$ and a term proportional to the group
size $\alpha_{\kappa}$.

Under the assumption $\mu=\rho$, we consider $a=\mathbb{E}\left(\chi_{1}\chi_{2}\right)$
as a function of the measure $\mu$. So $a:\mathcal{M}_{\leq1}\left(\left[-1/2,1/2\right]\right)\rightarrow\mathbb{R}_{+}$,
where $\mathcal{M}_{\leq1}\left(\left[-1/2,1/2\right]\right)$ is
the set of all sub-probability measures on $\left[-1/2,1/2\right]$.
We will also write $\mathcal{M}_{1}\left(\left[-1/2,1/2\right]\right)$
for the set of all probability measures. Similarly, $r$ is a function
$r:\mathcal{M}_{\leq1}\left(\left[-1/2,1/2\right]\right)\rightarrow\mathbb{R}_{+}$.
To show the theorem, we consider the cases of discrete and continuous
measures separately first and then show the general case.
\begin{prop}
\label{prop:mu_rho_discr}If $\mu$ is discrete, then we have $0\leq a\left(\mu\right)<1/3$.
The supremum over all discrete measures of $a\left(\mu\right)$ is
$1/3$. Within the class of discrete measures with at most $n$ points
belonging to $\textup{supp }\mu$, we have
\[
a\left(\mu\right)\leq\begin{cases}
\frac{\left(n-2\right)\left(n+2\right)}{3n^{2}}, & n\textup{ even},\\
\frac{\left(n-1\right)\left(n+1\right)}{3n^{2}}, & n\textup{ odd}.
\end{cases}
\]
\end{prop}

For measures $\mu$ with no atoms, we have
\begin{prop}
\label{prop:mu_rho_cont}If $\mu\in\mathcal{M}_{\leq1}\left(\left[-1/2,1/2\right]\right)$
has no atoms, then $a\left(\mu\right)\leq1/3$. If $\mu\in\mathcal{M}_{1}\left(\left[-1/2,1/2\right]\right)$,
then $a\left(\mu\right)=1/3$.
\end{prop}

For the remainder of this article, we express $m$ as the sum of two
terms:
\begin{align*}
m & =E\left|T_{1}\right|=E\left|Z+Y\right|=E\left(\,\textup{sgn\ensuremath{\left(Z+Y\right)}}\ensuremath{\cdot}\left(Z+Y\right)\right)=E\left(\,Z\,\textup{sgn\ensuremath{\left(Z+Y\right)}}\right)+E\left(\,Y\,\textup{sgn\ensuremath{\left(Z+Y\right)}}\right).
\end{align*}
The first of these summands equals $r$. The second one, we will call
$s$ from now on. If $\mu=\rho$, then of course $r=s$, and the term
$r-am$ equals $r\left(1-2a\right)$. For the proof of
these results, we need the following auxiliary lemma:
\begin{lem}
\label{lem:mu_rho_a_r}We can express the magnitudes $a\left(\mu\right)$
and $r\left(\mu\right)$ as
\begin{align*}
a\left(\mu\right) & =2\int_{\left(0,1/2\right]}\left(\mu\left(-z,z\right]\right)^{2}\mu\left(\textup{d}z\right),\quad r\left(\mu\right)=2\int_{\left(0,1/2\right]}z\mu\left(-z,z\right]\mu\left(\textup{d}z\right).
\end{align*}
\end{lem}

\begin{cor}
\label{cor:a_r_0}The terms $a$ and $r$ equal $0$ if and only if $\mu=\delta_{0}$.
\end{cor}

This follows easily from the representation of $a$ and $r$ given
in Lemma \ref{lem:mu_rho_a_r}.

The statements in this section are proved in the appendix.

\subsection{Non-Negativity of $w$ in Additive Collective Bias Models with $\mu\protect\neq\rho$}

In this section, we will not assume the two measures $\mu$ and $\rho$ are equal. The random variables  $Z,Y_1,\ldots,Y_M$ are all independent and $Y_1,\ldots,Y_M$ are i.i.d.$\!\,$ copies of a random variable $Y$ that follows a distribution according to $\rho$. As we already know, $r-am<0$ is possible in this case. We will give conditions under which this does not
happen. Analogously to Lemma \ref{lem:mu_rho_a_r}, we have these
representations of $a,r$, and $s$:
\begin{lem}
\label{lem:mu_neq_rho_a_r_s}We can express the magnitudes $a,r$,
and $s$ as
\begin{align*}
a & =2\int_{\left(0,1/2\right]}\left(\rho\left(-z,z\right]\right)^{2}\mu\left(\textup{d}z\right),\quad r=2\int_{\left(0,1/2\right]}z\rho\left(-z,z\right]\mu\left(\textup{d}z\right),\quad s=2\int_{\left(0,1/2\right]}y\mu\left(-y,y\right]\rho\left(\textup{d}y\right).
\end{align*}
\end{lem}

First we note that if the group modifiers override the central bias
almost surely, then the groups are independent (but the voters within
each group are still positively correlated!). In this case, the optimal
weights are proportional to the group sizes.
\begin{prop}
If $\left|Y\right|$ almost surely dominates $\left|Z\right|$, then
we have $a=r=0$ and $r-am=0$.
\end{prop}

This easily follows from Lemma \ref{lem:mu_neq_rho_a_r_s}.
\begin{rem}
If, instead, $\left|Z\right|$ almost surely dominates $\left|Y\right|$,
then the CBM is tightly correlated. We note, however, that in that case any set of weights
is optimal, among them weights proportional to the group sizes.
\end{rem}

Now we turn to first order stochastic dominance which is a weaker
form of the general concept of stochastic dominance.

\begin{defn}
We say that a random variable $X_{1}$ first order stochastically
dominates a random variable $X_{2}$ if, for all $x\in\mathbb{R}$,
$P\left(X_{1}\leq x\right)\leq P\left(X_{2}\leq x\right)$ holds.
We will write $X_{1}\succ X_{2}$ for this relation and FOSD for first
order stochastic dominance.
\end{defn}

This is weaker than almost sure dominance as it is possible to have
$X_{1}$ first order stochastically dominate $X_{2}$ without $X_{1}>X_{2}$
holding almost surely. We have the following sufficient conditions
for the non-negativity of the optimal weights:
\begin{prop}
\label{prop:FOSD_suff}If $\left|Z\right|\succ\left|Y\right|$ and
$a\leq1/2$, then $r-am\geq0$. If $\left|Y\right|\succ\left|Z\right|$
and $s\leq2r$, then $r-am\geq0$.
\end{prop}

The next idea is to assume that the measures $\mu$ and $\rho$ assign
similar probabilities to each event.
\begin{prop}
\label{prop:ribbon}Suppose there are constants $c,C>0$ such that,
for all measurable sets $A$,
\[
c\rho A\leq\mu A\leq C\rho A
\]
holds. Then each of the following two conditions is individually sufficient
for $r-am\geq0$:

\[
\text{1}.\;c\geq\frac{C^{2}}{3-C},C<3,\quad\text{2}.\;C\leq c\left(3c^{2}-1\right).
\]

If we assume additionally that $c=1/C$, then a sufficient condition
for $r-am\geq0$ is given by
\[
a\leq\frac{1}{1+C^{2}}.
\]
\end{prop}

Earlier we saw that if both $\mu$ and $\rho$ are uniform distributions (we will write $\mathcal{U}$ for a uniform distribution)
on symmetric intervals around the origin, $r-am\geq0$ holds. We can
generalise this result as follows:
\begin{prop}
\label{prop:rho_unif_mu_any}Let $\rho=\mathcal{U}\left[-1/2,1/2\right]$
and $\mu\in\mathcal{M}_{1}\left(\left[-1/2,1/2\right]\right)$. Then
$r-am\geq0$ is satisfied and $r=am$ if and only if $\mu=1/2\left(\delta_{-1/2}+\delta_{1/2}\right)$.
\end{prop}

\begin{rem}
Since $\mu=1/2\left(\delta_{-1/2}+\delta_{1/2}\right)$ implies that
$\left|Z\right|$ almost surely dominates $\left|Y\right|$, we can
disregard this case. Thus, this proposition implies that for $\rho=\mathcal{U}\left[-1/2,1/2\right]$
the optimal weights are given by a constant and a proportional part.
We also note here that $\rho$ being uniform on the entire interval
$\left[-1/2,1/2\right]$ is important. For every $0<\gamma<1/2,\rho=\mathcal{U}\left[-\gamma,\gamma\right]$,
there is a $\mu$ such that $r-am$ is negative.

The last result in this section concerns a case where $\rho$ is some
symmetric measure and $\mu$ is a contracted version of $\rho$ onto
some shorter interval $\left[-c/2,c/2\right]$ for some $0<c<1$. Hence, the global bias tends to be weaker than the local bias modifier.
For the rest of this section, assume the following conditions hold
\end{rem}

\begin{ass}\label{def:mu_contraction}
\begin{enumerate}
\item $\rho$ has no atoms.
\item There is a function $g:\left[0,\infty\right)\rightarrow\mathbb{R}$
with the property that $\rho\left(0,xy\right)=g\left(x\right)\rho\left(0,y\right)$
holds for all $x\geq0$ and all $y\in\left[0,1/2\right]$ such that
$xy\leq1/2$.
\item $\mu\left(cA\right)=\rho A$ for some fixed $0<c<1$ and all measurable
$A$.
\end{enumerate}
\end{ass}

The second point is a homogeneity condition. The last point is the
aforementioned contraction property.

Let $F_{\rho}$ be the distribution function of the sub-probability
measure $\rho\,\big|\!\left[0,1/2\right]$, i.e.$\!\,$ $\rho$ constrained to the subspace $[0,1/2]$. Note that due to property
1 above, $\rho\left[0,1/2\right]=1/2$, and hence $F_{\rho}\left(0\right)=0$
and $F_{\rho}\left(1/2\right)=1/2$.

The three properties in Assumptions \ref{def:mu_contraction}  already determine that the measures $\rho$
and $\mu$ belong to a two-parameter family indexed by $\left(t,c\right)\in\left(0,\infty\right)\times\left(0,1\right)$.
\begin{lem}
\label{lem:g_F_properties}If the second condition in Assumptions \ref{def:mu_contraction}
is satisfied, then
\begin{enumerate}
\item For all $y\in\left[0,1\right]$, $g\left(y\right)=2F_{\rho}\left(y/2\right)$.
\item $\rho$ has no atoms, unless $\rho=\delta_{0}$.
\item $g$ is multiplicative: for all $x,y\geq0$, $g\left(xy\right)=g\left(x\right)g\left(y\right)$.
\item $F_{\rho}$ has the form $F_{\rho}\left(y\right)=2^{t-1}y^{t}$ for
some fixed $t\geq0$.
\end{enumerate}
\end{lem}

\begin{rem}
If $t=0$ in the last point of the lemma, then $\rho=\mu=\delta_{0}$
and all voters are independent. We avoid this case by specifying the
first condition in Assumptions \ref{def:mu_contraction}.
\end{rem}

Now we state the theorem concerning the sign of the term $r-am$.
\begin{thm}
\label{thm:fam_rho_mu}Let the conditions stated in Assumptions \ref{def:mu_contraction}
hold for $\rho$ and $\mu$. Then, by the last lemma, $F_{\rho}\left(y\right)=2^{t-1}y^{t}$.
If $0<t<1$, then there is a unique $c_{0}\in\left(0,1\right)$ such
that, for all $c\in\left(0,c_{0}\right)$, $r-am$ is negative, and, for
all $c\in\left[c_{0},1\right)$, $r-am\geq0$ with equality if and
only if $c=c_{0}$. If $t\geq1$, then $r>am$.

The critical point $c_{0}$ for the regime $t\in\left(0,1\right)$
satisfies $\lim_{t\nearrow1}c_{0}=0$.
\end{thm}

\begin{rem}
Note that $t=1$ is the case of the uniform distributions $\rho=\mathcal{U}\left[-\gamma,\gamma\right],\mu=\mathcal{U}\left[-\beta,\beta\right]$
with $\gamma=1/2$ and $0<\beta=c/2<1/2$.
\end{rem}

The proofs of these statements can be found in the appendix.

\section{Extensions}\label{sec:extensions}

An obvious extension to the general CBM framework is to allow different
conditional distributions $\rho_{\lambda}^{z}$ for the different groups to
account for more strongly or more weakly correlated groups. More precisely,
\begin{align}
   \IP\big(\ux_{1},\ux_{2},\ldots,\ux_{M}\big)~
=~&\int  \left( \int P_{t_{1}}(\ux_{1})\,\rho_{1}^{z}(\textup{d}t_{1})
\cdots\int P_{t_{M}}(\ux_{M})
\,\rho_{M}^{z}(\textup{d}t_{M}) \right)\;\mu(\textup{d}z)\label{eq:extCBM}
\end{align}
A large part of the analysis in Sections \ref{sec:demdef} and \ref{sec:asymp} can be done for this more general case as well. In fact, with the definitions \eqref{eq:A}, \eqref{eq:b}, and \eqref{eq:AN}, the optimal weights $w$ for this model again satisfy
\begin{align*}
   A_{N}\,w~=~b_{N}.
\end{align*}

In the limit $N\to\infty $, using the same technique as in Section \ref{sec:asymp}, we obtain
\begin{align*}
   A\,w~=~b,
\end{align*}
with
\begin{align*}
  A_{\lambda\nu}~&=~ \left\{
                       \begin{array}{cl}
                         1, & \hbox{if $\lambda=\nu $,} \\
                         \big\langle\, d(\rho_{\lambda})\, d(\rho_{\nu}) \,\big\rangle , & \hbox{if $\lambda\not=\nu $,}
                       \end{array}
                     \right.
  \\
  b_{\nu}~&=~\big(\langle \om_{1}(\rho_{\nu})\rangle-\langle m_{1}(\rho_{\nu})d(\rho_{\nu}) \rangle\big)\,\alpha_{\nu}\;+\;\sum_{\lambda=1}^{M}\,\big\langle\, m_{1}(\rho_{\lambda})\, d(\rho_{\nu}))\, \big\rangle\,\alpha_{\lambda}.
\end{align*}
As in Proposition \ref{prop:posdef}, it is easy to see, that the matrix $A$ is positive semi-definite. Moreover, we show

\begin{thm}\label{thm:extpos}
 The matrix $A$ is positive definite, and hence invertible, if  $\big\langle\, d(\rho_{\lambda})^{2} \,\big\rangle<1$ for all but possibly one $\lambda$.
\end{thm}
\begin{proof}
  For $x\in\IR^{M}$, we compute
\begin{align}
   \big(x,A\,x\big)~&=~\sum_{\nu=1}^{M} \big(1-\big\langle\, d(\rho_{\nu})^{2} \,\big\rangle\big)\,x_{\nu}^{2}\;+\;
    \sum_{\nu,\lambda=1}^{M} \big\langle\, d(\rho_{\nu})\, d(\rho_{\lambda}) \,\big\rangle\,x_{\nu}\,x_{\lambda}\notag\\
&=~\sum_{\nu=1}^{M} \big(1-\big\langle\, d(\rho_{\nu})^{2} \,\big\rangle\big)\,x_{\nu}^{2}\;
+\;\left\langle \left(\sum_{\nu=1}^M x_{\nu} d(\rho_{\nu})\right)^{2} \right\rangle\label{eq:posdef}
\end{align}
Both terms in \eqref{eq:posdef} are non-negative. If  $\big\langle\, d(\rho_{\nu})^{2} \,\big\rangle<1$ for \emph{all} $\nu $, the first sum in \eqref{eq:posdef} is strictly positive for $x\not=0 $, hence $A$ is positive definite in this case. If all but one $\nu^*$ have this property, the conclusion follows from considering that the first summand is 0 if and only if $x_{\nu^*}$ is the only coordinate of $x$ different than 0. But in that case, the second summand is positive.

\end{proof}

We have two partial converses to Theorem \ref{thm:extpos}.

\begin{prop}
   If $|\,\langle \, d(\rho_{\nu})\,d(\rho_{\lambda})\,\rangle\,|=1 $ for some $\nu\not=\lambda $, then
   the matrix $A$ is \emph{not} invertible.
\end{prop}

\begin{proof}
   If $|\,\langle \, d(\rho_{\nu})\,d(\rho_{\lambda})\,\rangle\,|=1 $, then an application of the 
   Cauchy-Schwarz inequality shows that $\langle \, d(\rho_{\nu})^{2}\,\rangle = \langle\,d(\rho_{\lambda})^{2}\,\rangle=1 $.
      This implies that the matrix
   \begin{align*}
      \begin{pmatrix} \langle \, d(\rho_{\nu})^{2}\,\rangle & \langle \, d(\rho_{\nu})\,d(\rho_{\lambda})\,\rangle\\
      \langle \, d(\rho_{\nu})\,d(\rho_{\lambda})\,\rangle & \langle\,d(\rho_{\lambda})^{2}\,\rangle
      \end{pmatrix}
   \end{align*}
   is not invertible, hence $A $ is not invertible.
\end{proof}

\begin{prop}
Let for $\mu$-almost all $z$ $\textup{sgn }d\left(\rho_{\lambda}\right)=\textup{sgn } d\left(\rho_{\nu}\right), \lambda,\nu=1,\ldots,M$. Then the matrix $A$ being positive definite implies
$\big\langle\,d(\rho_{\lambda})^{2}\,\big\rangle<1$ for all but possibly
one $\lambda$.
\end{prop}

\begin{rem}
Additive CBMs satisfy the condition of all $d\left(\rho_{\lambda}\right)$
having the same sign for $\mu$-almost all $z$ owing to the symmetry
of each measure $\rho_{\lambda}$. Multiplicative CBMs have this property
if we assume a certain asymmetry for each $\rho_{\lambda}$: If, for all
$\lambda$, $\rho_{\lambda}\left(0,1\right]>1/2$, or, for all $\lambda$, $\rho_{\lambda}\left(0,1\right]<1/2$,
then the condition holds. This can be interpreted as all group bias
modifiers tending to reinforce the global bias, or, to the contrary,
all tending to go against the global bias. Note, however, that this
is a far weaker tendency than required by tight correlation.
\end{rem}

\begin{proof}
Assume there are two distinct indices $\lambda_{1}$ and $\lambda_{2}$ such
that $\langle\,d(\rho_{\lambda_{1}})^{2}\,\rangle,\langle\,d(\rho_{\lambda_{2}})^{2}\,\rangle=1$.
We show that $A$ is not positive definite. Define an $x\in\mathbb{R}^{M}$
by setting $x_{\lambda_{1}}=1,x_{\lambda_{2}}=-1$, and all other entries equal
to 0. We calculate 
\begin{align*}
\big(x,A\,x\big)~ & =~\big(1-\big\langle\,d(\rho_{\lambda_{1}})^{2}\,\big\rangle\big)\,+\big(1-\big\langle\,d(\rho_{\lambda_{2}})^{2}\,\big\rangle\big)\;+\;\left\langle\big(d(\rho_{\lambda_{1}})-d(\rho_{\lambda_{2}})\big)^{2}\right\rangle\\
 & =\left\langle d(\rho_{\lambda_{1}})^{2}\right\rangle -2\left\langle d(\rho_{\lambda_{1}})d(\rho_{\lambda_{2}})\right\rangle +\left\langle d(\rho_{\lambda_{2}})^{2}\right\rangle 
\end{align*}
By assumption, $d(\rho_{\lambda_{1}})^{2}$ and $d(\rho_{\lambda_{2}})^{2}$
are 1 $\mu$-almost surely. Hence, $d(\rho_{\lambda_{1}})$ and $d(\rho_{\lambda_{2}})$
are 1 in absolute value and they have the same sign $\mu$-almost
surely. Thus, the term $\big\langle d(\rho_{\lambda_{1}})d(\rho_{\lambda_{2}})\big\rangle$
equals 1.
\end{proof}
We will now consider a scenario in which there are two clusters of
groups -- think of them as parts of the overall population that tend
to vote together. While voters within clusters tend to hold the same
opinion, we will assume that there is antagonism between the two clusters.

Let groups $1,\ldots,M_{1}$ belong to cluster $C_{1}$ and $M_{1}+1,\ldots,M$
to cluster $C_{2}$. We set $M_{2}:=M-M_{1}$. The fraction of the
overall population belonging to groups in each cluster will be called
$\eta_{i}:=\sum_{\lambda\in C_{i}}\alpha_{\lambda},i=1,2$. The conditional
distributions $\rho_{\lambda}^{z}$ are identical within each cluster:
In $C_{i}$, all groups follow $\rho_{i}^{z}$. To obtain antagonistic
behaviour, we will assume that 
\begin{equation}
\rho_{1}^{z}=\rho_{2}^{-z}\label{eq:antisymmetry}
\end{equation}
holds for all $z$.

Now we have to distinguish the quantities
\[
r_{ij}:=\left\langle m_{1}\left(\rho_{i}\right)d\left(\rho_{j}\right)\right\rangle ,\quad m_{i}:=\left\langle \overline{m}_{1}\left(\rho_{i}\right)\right\rangle ,\quad i,j=1,2.
\]

However, due to the antisymmetry condition (\ref{eq:antisymmetry}),
we have the following equalities:
\begin{lem}
\label{lem:m,r,i,j}Under the assumptions presented above, we have
\begin{align*}
r :=r_{ii}=-r_{ij}\quad and \quad m :=m_{i}
\end{align*}
for all $i,j=1,2,i\neq j$.
\end{lem}

\begin{proof}
We omit the short calculation that yields the result.
\end{proof}
The covariance matrix $A$ has block form
\[
A=
\left(\begin{array}{cc}
A_{1} & B\\
B^{T} & A_{2}
\end{array}\right),
\]
where the $A_{i}$ are the covariance matrices of the groups belonging
to cluster $i$. They have the form we know from \eqref{eq:Ainf},
i.e. diagonal entries equal 1 and off-diagonal entries $0<a<1$. Of
course, $A_{i}\in\mathbb{R}^{M_{i}\times M_{i}}$. The matrix $B$
holds the covariances between groups of different clusters. Due to
(\ref{eq:antisymmetry}), all entries of $B$ are equal to $-a$.

We invert $A$ and obtain
\[
\left(A^{-1}\right)_{\lambda\nu}=\frac{1}{D}\begin{cases}
1+\left(M-2\right)a, & \lambda=\nu,\\
-a, & \lambda,\nu\in C_{i},\quad i=1,2,\quad\lambda\neq\nu,\\
a, & \lambda\in C_{i},\quad\nu\in C_{j},\quad i,j=1,2,\quad i\neq j,
\end{cases}
\]
where $D=\left(1-a\right)\left(1+\left(M-1\right)a\right)$. Note
that the entries within clusters are identical to those given in \eqref{eq:Ainv} for the model with identical conditional distributions $\rho_{\lambda}^{z}$.

Using Lemma \ref{lem:m,r,i,j}, we calculate the entries $\lambda\in C_i,i=1,2,$
of $b$:
\[
b_{\lambda}=\left(m-r\right)\alpha_{\lambda} + r\left(\eta_i - \eta_j\right).
\]
In the formula above, the index $j$ is the cluster $\lambda$ does \emph{not} belong to.
Now a lengthy but straightforward calculation yields the optimal
weights for each group $\nu$:
\begin{thm}
Let $\lambda$ be a group in cluster $i=1,2$ and let $j$ be the other cluster. Then the optimal weight of group $\lambda$ is given by
\begin{align*}
w_{\lambda} =\left(A^{-1}b\right)_{\lambda} = D_{1}\alpha_{\lambda}+D_{2},
\end{align*}
where the coefficients are
\begin{align*}
D_1  =\frac{m-r}{1-a}\quad and \quad  D_2  =\frac{\left(r-am\right)\left(\eta_{i}-\eta_{j}\right)}{\left(1-a\right)\left(1+\left(M-1\right)a\right)}.
\end{align*}

\end{thm}

$D_1$ is equal to the coefficient $C_1$ given in \eqref{eq:C1} for the model with identical conditional distributions. We note that if both clusters have exactly half the overall population,
then $D_{2}$ vanishes, and the optimal weights are proportional to
the population of each group. If the two clusters represent different
proportions of the overall population, then $w_{\lambda}$ is the sum
of a proportional term $D_{1}\alpha_{\lambda}$ and a constant $D_{2}$.
If $\lambda$ belongs to the larger of the two clusters, then $D_{2}$
has the same sign as the coefficient $C_{2}$ in \eqref{eq:C2} in the
identical conditional distribution model, and $D_{2}$ is a rescaled version of $C_2$ by the factor $\left(\eta_{i}-\eta_{j}\right)$.
If $\lambda$ belongs to the smaller of the two clusters, then $D_{2}$
has the opposite sign compared to $C_{2}$ in \eqref{eq:C2} and it is
once again rescaled.

\section{Conclusion}\label{sec:conclusion}

We have defined and analysed a multi-group version of the CBM which allows for correlated voting across group boundaries. This CBM was then applied to the problem of calculating the optimal weights in a two-tier voting system. By the term `optimal weights', we mean those council weights which minimise the democracy deficit, i.e.\! the expected quadratic deviation of the council vote from a hypothetical referendum over all possible issues which can be voted on. The main findings in this paper are:

\begin{itemize}
\item{We determined the asymptotic behaviour of the CBM in Theorem \ref{thm:convthm}. The theorem states that the global bias measure $\mu$ and the group bias measure $\rho$ describe the limiting distribution of the normalised voting margins.}
\item{We distinguished the tightly correlated case from its complement. We characterised tight correlation in Proposition \ref{prop:tight} in terms of the bias measures. Tight correlation means intuitively that there is perfect positive correlation between the different group votes in the council. This leads to non-unique optimal weights, as the assignation of weights does not matter if all groups vote alike anyway. We gave a sufficient condition for non-tight correlation between groups called `sufficient randomness'. This criterion states that all possible council votes occur with positive probability.}
\item{In the non-tightly correlated case, we showed that the optimal council weights are uniquely determined in Theorem \ref{thm:weights}. The optimal weights are given by the sum of a constant term equal for all groups and a summand which is proportional to each group's population as stated in Theorem \ref{thm:asweights}.}
\item{We analysed the optimal weights' properties and showed that there are cases in which these weights are negative for the smallest groups. This is due to the fact that while the coefficient of the group's size is positive, the constant term can have any sign, depending on the bias measures. We gave sufficient conditions for the non-negativity of the optimal weights as well as examples in which the weights are negative in Section \ref{sec:nonneg}.}
\end{itemize}

\section*{Appendix}

\subsection*{Proof of Proposition \ref{prop:mu_rho_discr}}

We prove the claim for $n=2k+1$. The case of even $n$ can be shown
analogously. We prove by induction on $k$ that
\begin{equation}
a\left(\mu\right)\leq\frac{2k\left(2k+2\right)}{3\left(2k+1\right)^{2}}, \label{eq:a_discr_inequ}
\end{equation}
 with equality if $\mu$ is chosen to be the uniform distribution
on the $2k+1$ points conforming the support of $\mu$.

\uline{Base case:} Let $k=1$. Then the support of $\mu$ consists
of three points: 0 and two points $-x_{1},x_{1}$ such that $0<x_{1}\leq1/2$.
The measure $\mu$ is given by $\beta_{0}\delta_{0}+\beta_{1}\left(\delta_{-x_{1}}+\delta_{x_{1}}\right)$
and the constants satisfy $\beta_{0}+2\beta_{1}=1$. Set $\beta:=\beta_{1}$.
To show the upper bound (\ref{eq:a_discr_inequ}), we solve the maximisation
problem $\max_{\beta}a\left(\mu\right)$. The first order condition
is
\[
\left(1-\beta\right)^{2}-2\beta\left(1-\beta\right)=0,
\]
which has two solutions: $\beta=1$ and $\beta=1/3$. The second order
condition shows that $\beta=1$ minimises $a\left(\mu\right)$ and
$\beta=1/3$ maximises it. So, for $k=1$, the uniform distribution
maximises $a\left(\mu\right)$ and, for the uniform distribution $\mu_{3}$
on $\left\{ -x_{1},0,x_{1}\right\} $,
\[
a\left(\mu_{3}\right)=\frac{8}{27}=\frac{2k\left(2k+2\right)}{3\left(2k+1\right)^{2}},
\]
and the upper bound (\ref{eq:a_discr_inequ}) holds with equality.

\uline{Induction step:} Assume that for some $k\in\mathbb{N}$
and all sets $\left\{ -x_{k},\ldots,0,\ldots,x_{k}\right\} ,0<x_{1}<\cdots<x_{k}\leq1/2,$
the uniform distribution $\mu_{2k+1}$ maximises $a\left(\mu\right)$
and $a\left(\mu_{2k+1}\right)=\frac{2k\left(2k+2\right)}{3\left(2k+1\right)^{2}}$.
We add another point $1/2\geq x_{k+1}>x_{k}$ (if $x_{k}=1/2$, then
relabel the last two points) with probability $1\geq\eta\geq0$ and
solve the maximisation problem
\[
\max_{\mu,\eta}\;a\left(\left(1-2\eta\right)\mu+\eta\left(\delta_{-x_{k+1}}+\delta_{x_{k+1}}\right)\right),
\]
where $\mu$ is any symmetric probability measure on $\left\{ -x_{k},\ldots,0,\ldots,x_{k}\right\} $.
Set $\nu:=\left(1-2\eta\right)\mu$\\$+\,\eta\left(\delta_{-x_{k+1}}+\delta_{x_{k+1}}\right)$
and we calculate
\begin{align*}
 & \qquad a\left(\nu\right)/2=\int_{\left(0,1/2\right]}\left(\nu\left(-z,z\right]\right)^{2}\nu\left(\textup{d}z\right)\\
 & =\int_{\left(0,x_{k}\right]}\left(\left(1-2\eta\right)\mu\left(-z,z\right]\right)^{2}\left(1-2\eta\right)\mu\left(\textup{d}z\right)+\int_{\left(0,x_{k}\right]}\left(\left(1-2\eta\right)\mu\left(-z,z\right]\right)^{2}\eta\delta_{x_{k+1}}\left(\textup{d}z\right)\\
 & \quad+\int_{\left(x_{k},1/2\right]}\left(\left(1-2\eta\right)+\eta\delta_{x_{k+1}}\left(-z,z\right]\right)^{2}\left(1-2\eta\right)\mu\left(\textup{d}z\right)+\int_{\left(x_{k},1/2\right]}\left(\left(1-2\eta\right)+\eta\delta_{x_{k+1}}\left(-z,z\right]\right)^{2}\eta\delta_{x_{k+1}}\left(\textup{d}z\right).
\end{align*}
The second summand is 0 because $\delta_{x_{k+1}}\left(0,x_{k}\right]=0$.
The third summand is 0 due to $\mu\left(x_{k},1/2\right]=0$. We continue
\begin{align*}
 & \qquad a\left(\nu\right)/2=\\
 & =\int_{\left(0,x_{k}\right]}\left(\left(1-2\eta\right)\mu\left(-z,z\right]\right)^{2}\left(1-2\eta\right)\mu\left(\textup{d}z\right)+\int_{\left(x_{k},1/2\right]}\left(\left(1-2\eta\right)+\eta\delta_{x_{k+1}}\left(-z,z\right]\right)^{2}\eta\delta_{x_{k+1}}\left(\textup{d}z\right)\\
 & =\left(1-2\eta\right)^{3}\int_{\left(0,1/2\right]}\left(\mu\left(-z,z\right]\right)^{2}\mu\left(\textup{d}z\right)+\eta\left(\left(1-2\eta\right)+\eta\delta_{x_{k+1}}\left(-x_{k+1},x_{k+1}\right]\right)^{2}\\
 & =\left(1-2\eta\right)^{3}a\left(\mu\right)/2+\eta\left(1-\eta\right)^{2}.
\end{align*}
As we see, $\mu$ and $\eta$ can be chosen independently of each
other to maximise $a\left(\nu\right)$. By assumption, the maximising
$\mu$ is the uniform distribution on $\left\{ -x_{k},\ldots,0,\ldots,x_{k}\right\} $
$\mu_{2k+1}$. Hence,
\begin{align*}
\max_{\nu}\;a\left(\nu\right) & =\max_{\eta}\;\left(1-2\eta\right)^{3}a\left(\mu_{2k+1}\right)+2\eta\left(1-\eta\right)^{2}.
\end{align*}
Since $a\left(\mu_{2k+1}\right)$ is independent of the choice of
$\eta$, the first order condition is
\[
3\left(1-2\eta\right)^{2}a\left(\mu_{2k+1}\right)=\left(1-\eta\right)\left(1-3\eta\right).
\]
The solutions of this quadratic equation are
\[
\eta=\frac{1}{3}\frac{2-3a\left(\mu_{2k+1}\right)}{1-2a\left(\mu_{2k+1}\right)}\pm\frac{1}{3}\sqrt{\left(\frac{2-3a\left(\mu_{2k+1}\right)}{1-2a\left(\mu_{2k+1}\right)}\right)^{2}-\frac{3}{2}\frac{2-3a\left(\mu_{2k+1}\right)}{1-2a\left(\mu_{2k+1}\right)}}.
\]
By substituting $a\left(\mu_{2k+1}\right)=\frac{2k\left(2k+2\right)}{3\left(2k+1\right)^{2}}$,
we see that the root with the negative sign gives a negative $\eta$.
The positive root is $\eta=\frac{1}{2k+3}$. This implies that the
maximising measure $\nu$ on $\left\{ -x_{k+1},\ldots,0,\ldots,x_{k+1}\right\} $
is the uniform distribution $\mu_{2(k+1)+1}$. This concludes the
proof by induction that for finitely many points in the support, the
uniform distribution maximises $a$ and this maximum is given by the
upper bound in (\ref{eq:a_discr_inequ}).

Next we show that for discrete measures with infinite support the
upper bound $1/3$ holds as well. If $\left|\textup{supp }\mu\right|=\infty$,
then $\mu$ is of the form $\beta_{0}\delta_{0}+\sum_{i=1}^{\infty}\beta_{i}\left(\delta_{-x_{i}}+\delta_{x_{i}}\right)$,
where $x_{i}>0$ for all $i\in\mathbb{N}$. Set for each $n\in\mathbb{N}$
$\nu_{n}:=\beta_{0}\delta_{0}+\sum_{i=1}^{n}\beta_{i}\left(\delta_{-x_{i}}+\delta_{x_{i}}\right)$.
To obtain a contradiction, suppose that $a\left(\mu\right)>1/3$ and
set $\tau:=a\left(\mu\right)-1/3>0$. The sequence $\left(a\left(\nu_{n}\right)\right)_{n}$
is monotonically increasing:
\begin{align*}
a\left(\nu_{n}\right) & =2\int_{\left(0,1/2\right]}\left(\nu_{n}\left(-z,z\right]\right)^{2}\nu_{n}\left(\textup{d}z\right)\leq2\int_{\left(0,1/2\right]\cap\left\{ x_{1},\ldots,x_{n}\right\} }\left(\nu_{n+1}\left(-z,z\right]\right)^{2}\nu_{n}\left(\textup{d}z\right)\\
 & \quad+2\left(\nu_{n+1}\left(-x_{n+1},x_{n+1}\right]\right)^{2}\beta_{n+1}=a\left(\nu_{n+1}\right).
\end{align*}
For any $\varepsilon>0$, there is some $m\in\mathbb{N}$ such that
for all measurable sets $A\subset\left[-1/2,1/2\right]$ the inequality
$\mu A-\nu_{m}A<\varepsilon$ holds. So
\begin{align*}
a\left(\nu_{m}\right) & =2\int_{\left(0,1/2\right]}\left(\nu_{m}\left(-z,z\right]\right)^{2}\nu_{m}\left(\textup{d}z\right)>2\int_{\left(0,1/2\right]}\left(\mu\left(-z,z\right]-\varepsilon\right)^{2}\left(\mu-\varepsilon\right)\left(\textup{d}z\right)\\
 & =a\left(\mu\right)-2\varepsilon\int_{\left(0,1/2\right]}\left(\mu\left(-z,z\right]\right)^{2}\textup{d}z-4\varepsilon\int_{\left(0,1/2\right]}\mu\left(-z,z\right]\mu\left(\textup{d}z\right)\\
 & \quad+4\varepsilon^{2}\int_{\left(0,1/2\right]}\mu\left(-z,z\right]\textup{d}z+2\varepsilon^{2}\int_{\left(0,1/2\right]}\mu\left(\textup{d}z\right)-2\varepsilon^{3}\int_{\left(0,1/2\right]}\textup{d}z.
\end{align*}
By letting $\varepsilon$ go to 0, we see that $a\left(\nu_{n}\right)\nearrow a\left(\mu\right)$
and there is an $n\in\mathbb{N}$ such that $a\left(\nu_{n}\right)>a\left(\mu\right)-\tau/2>1/3$.
This is a contradiction because the cardinality $\left|\textup{supp }\nu_{n}\right|$
equals $2n+1$ and therefore $a\left(\nu_{n}\right)\leq1/3$.

Next we note that $\sup a\left(\mu\right)$ over all discrete probability
measures $\mu$ is $1/3$. This is easy to see because of the following
facts:
\begin{lem}
The sequences $\left(\frac{\left(n-2\right)\left(n+2\right)}{3n^{2}}\right)_{n\textup{ even}}$
and $\left(\frac{\left(n-1\right)\left(n+1\right)}{3n^{2}}\right)_{n\textup{ odd}}$
are monotonically increasing and their limit is equal to $1/3$.
\end{lem}

As we have proved, for uniform distributions $\mu_{n}$, $a\left(\mu_{n}\right)$
is equal to one of these expressions depending on the parity of $n$.
From this lemma, it follows that by choosing a discrete uniform distribution
on either an even- or odd-numbered support we can get arbitrarily
close to $1/3$. This concludes the proof of Proposition \ref{prop:mu_rho_discr}.

Next we prove the result for continuous measures.

\subsection*{Proof of Proposition \ref{prop:mu_rho_cont}}

Let $\mu\in\mathcal{M}_{1}\left(\left[-1/2,1/2\right]\right)$ have
no atoms. We show $a\left(\mu\right)=1/3$ by approximating $a\left(\mu\right)/2$
by a sum and then prove that the sum in question is a Riemann sum
of the function $x\mapsto4x^{2}$ on the interval $\left(0,1/2\right)$.

Let $\varepsilon>0$ be given. Then there is a partition $\mathcal{P}_{n}=\left(I_{1},\ldots,I_{n}\right)$
of $\left(0,1/2\right)$ with the property that for all $i=1,\ldots,n$
$\mu I_{i}<\varepsilon$. We assume the intervals are ordered from
left to right. It is possible to choose at most $n\leq\left\lceil 1/\varepsilon\right\rceil $
intervals for the partition $\mathcal{P}_{n}$. Then we define the
upper and lower sum
\begin{align*}
U\left(\mathcal{P}_{n}\right) & :=\sum_{i=1}^{n}\left(\sum_{j=1}^{i}2\mu I_{j}\right)^{2}\mu I_{i},\quad L\left(\mathcal{P}_{n}\right):=\sum_{i=1}^{n}\left(\sum_{j=1}^{i-1}2\mu I_{j}\right)^{2}\mu I_{i}.
\end{align*}
For each summand $i=1,\ldots,n$, we have
\begin{align}
\left|\left(\sum_{j=1}^{i}2\mu I_{j}\right)^{2}\mu I_{i}-\left(\sum_{j=1}^{i-1}2\mu I_{j}\right)^{2}\mu I_{i}\right| & \leq2\left|\sum_{j=1}^{i}2\mu I_{j}-\sum_{j=1}^{i-1}2\mu I_{j}\right|\mu I_{i}=2\cdot2\left(\mu I_{i}\right)^{2}<4\varepsilon^{2}.\label{eq:summand_diff}
\end{align}
In the inequality above, we used that for all $x,y\in\left[0,1\right]$
$\left|x^{2}-y^{2}\right|<2\left|x-y\right|$ holds. Also for all
$i=1,\ldots,n$ and all $z\in I_{i}$
\[
\sum_{j=1}^{i-1}2\mu I_{j}\leq\mu\left(-z,z\right]\leq\sum_{j=1}^{i}2\mu I_{j},
\]
and, therefore,
\begin{equation}
L\left(\mathcal{P}_{n}\right)\leq\int_{\left(0,1/2\right]}\left(\mu\left(-z,z\right]\right)^{2}\mu\left(\textup{d}z\right)\leq U\left(\mathcal{P}_{n}\right).\label{eq:sum_inequ}
\end{equation}
Due to (\ref{eq:summand_diff}) and (\ref{eq:sum_inequ}), we have
\begin{align*}
0 & \leq U\left(\mathcal{P}_{n}\right)-L\left(\mathcal{P}_{n}\right)\leq n\cdot4\varepsilon^{2}\leq\left\lceil 1/\varepsilon\right\rceil \cdot4\varepsilon^{2}\leq4\varepsilon\left(1+\varepsilon\right).
\end{align*}
This shows that the upper and lower sum approximate $a\left(\mu\right)/2$
well as we let the number of intervals in $\mathcal{P}_{n}$ go to
infinity.

The next step is to show $U\left(\mathcal{P}_{n}\right)$ is an upper
Riemann sum of the function $x\mapsto4x^{2}$. For the partition $\mathcal{P}_{n}$,
there is a corresponding partition $\mathcal{Q}_{n}=\left(J_{1},\ldots,J_{n}\right)$
in which the intervals are once again assumed to be ordered from left
to right and for each $i=1,\ldots,n$ the interval lengths $\left|J_{i}\right|$
equal $\mu I_{i}$. We define
\[
R\left(\mathcal{Q}_{n}\right):=\sum_{i=1}^{n}\sup_{x\in J_{i}}4x^{2}\cdot\left|J_{i}\right|.
\]
This is an upper Riemann sum of $x\mapsto4x^{2}$. On the other hand,
we have
\begin{align*}
R\left(\mathcal{Q}_{n}\right) & =\sum_{i=1}^{n}\left(2\sup_{x\in J_{i}}x\right)^{2}\left|J_{i}\right|=\sum_{i=1}^{n}\left(2\sup J_{i}\right)^{2}\left|J_{i}\right|=\sum_{i=1}^{n}\left(2\sum_{j=1}^{i}\left|J_{j}\right|\right)^{2}\left|J_{i}\right|\\
 & =\sum_{i=1}^{n}\left(\sum_{j=1}^{i}2\mu I_{j}\right)^{2}\mu I_{i}=U\left(\mathcal{P}_{n}\right)\searrow a\left(\mu\right)/2.
\end{align*}

Since $R\left(\mathcal{Q}_{n}\right)$ is an upper Riemann sum of
$x\mapsto4x^{2}$, $R\left(\mathcal{Q}_{n}\right)\searrow\int_{0}^{1/2}4x^{2}\textup{d}x=1/6$
holds as we let the number of intervals in the partition go to infinity
and we are done.

Finally, we show the case of a general probability measure.

\subsection*{Proof of Theorem \ref{thm:mu_rho}}

Let $\mu\in\mathcal{M}_{1}\left(\left[-1/2,1/2\right]\right)$ . We
can express $\mu$ as the sum of a discrete sub-probability measure
$\delta$ and a sub-probability measure $\gamma$ that has no atoms.
Both $\delta$ and $\gamma$ must satisfy the symmetry condition (\ref{eq:symmetry}).
Therefore, $\delta$ must have the form $\beta_{0}\delta_{0}+\sum_{i=1}^{\infty}\beta_{i}\left(\delta_{-x_{i}}+\delta_{x_{i}}\right)$.
Similarly to the proof of Proposition \ref{prop:mu_rho_discr}, we
truncate the sum to $\delta_{n}=\beta_{0}\delta_{0}+\sum_{i=1}^{n}\beta_{i}\left(\delta_{-x_{i}}+\delta_{x_{i}}\right)$
choosing $n$ large enough for a condition $\delta A-\delta_{n}A<\varepsilon$
to hold for all measurable sets $A$ and proceed with $\delta_{n}$
instead of $\delta$. Set $\nu:=\delta_{n}+\gamma$. Our strategy
is to show that if we remove one pair of the points $-x_{i},x_{i}$
from $\textup{supp }\delta_{n}$ and add the probability mass $2\delta\left\{ x_{i}\right\} $
to $\gamma$ as a uniform distribution on two small intervals around
$-x_{i},x_{i}$, we obtain a new measure $\nu^{\left(0\right)}$ and
we increase $a$: $a\left(\nu\right)<a\left(\nu^{\left(0\right)}\right)$.
So by removing the $2n+1$ points in $\textup{supp }\delta_{n}$ in
pairs (except for the origin where we remove a single point), we obtain
a monotonically increasing finite sequence $a\left(\nu^{\left(i\right)}\right)_{i=0,\ldots,n}$.
After $n+1$ steps, we have a sub-probability measure $\nu^{\left(n\right)}$
with no atoms and the bound $a\left(\nu^{\left(n\right)}\right)\leq1/3$
thus applies.

Let $x=x_{i}$ for some $i\in\left\{ 1,\ldots,n\right\} $ and set
$\alpha:=\delta\left\{ x\right\} >0$. Let $\varepsilon>0$ be given.
Then we choose $\eta>0$ with the properties
\begin{enumerate}
\item $2\left|\int_{\left(0,x-\eta\right]}\left(\nu\left(-z,z\right]\right)^{2}\nu\left(\textup{d}z\right)-\int_{\left(0,x\right)}\left(\nu\left(-z,z\right]\right)^{2}\nu\left(\textup{d}z\right)\right|<\varepsilon$,
\item $\left[x-\eta,x\right)\cap\textup{supp }\delta_{n}=\emptyset$,
\item $\gamma\left(x-\eta,x\right)<\varepsilon$,
\item $\nu\left(-x,x\right)-\nu\left(-\left(x-\eta\right),x-\eta\right]<\varepsilon$.
\end{enumerate}
Next define the sub-probability measure
\[
\pi_{\varepsilon}:=\delta_{n}-\alpha\left(\delta_{-x}+\delta_{x}\right)+\gamma-\gamma\left|\left(x-\eta,x\right)\right.+\alpha\left(\mathcal{U}\left(x-\eta,x\right)+\mathcal{U}\left(-x,-x+\eta\right)\right).
\]
Here $\mathcal{U}$ stands for a uniform distribution. We remove the
points $-x,x$ from $\delta_{n}$ as well as the continuous measure
$\gamma$ on the interval $\left(x-\eta,x\right)$ and add in the
probability mass $2\alpha$ on small intervals close to $-x$ and
$x$, respectively. Note that by property 2 above, $\gamma\left|\left(x-\eta,x\right)\right.=\nu\left|\left(x-\eta,x\right)\right.$.
Also, by 3, $\nu\left[-1/2,1/2\right]-\varepsilon<\pi_{\varepsilon}\left[-1/2,1/2\right]$.

We divide $a\left(\nu\right)$ into four summands
\begin{align*}
a\left(\nu\right) & =2\left[\int_{\left(0,x-\eta\right]}\left(\nu\left(-z,z\right]\right)^{2}\nu\left(\textup{d}z\right)+\int_{\left(x-\eta,x\right)}\left(\nu\left(-z,z\right]\right)^{2}\nu\left(\textup{d}z\right)+\right.\\
 & \quad\left.+\left(\nu\left(-x,x\right)+\alpha\right)^{2}\alpha+\int_{\left(x,1/2\right]}\left(\nu\left(-z,z\right]\right)^{2}\nu\left(\textup{d}z\right)\right].
\end{align*}
We define the terms
\begin{align*}
A_{\nu} & :=\int_{\left(0,x-\eta\right]}\left(\nu\left(-z,z\right]\right)^{2}\nu\left(\textup{d}z\right),\quad B_{\nu}:=\int_{\left(x-\eta,x\right]}\left(\nu\left(-z,z\right]\right)^{2}\nu\left(\textup{d}z\right),\quad C_{\nu}:=\int_{\left(x,1/2\right]}\left(\nu\left(-z,z\right]\right)^{2}\nu\left(\textup{d}z\right),
\end{align*}
and, analogously, we define $A_{\pi_{\varepsilon}},B_{\pi_{\varepsilon}},C_{\pi_{\varepsilon}}$
on the same intervals in each case. We note that

\[
B_{\nu}=\int_{\left(x-\eta,x\right)}\left(\nu\left(-z,z\right]\right)^{2}\nu\left(\textup{d}z\right)+\left(\nu\left(-x,x\right)+\alpha\right)^{2}\alpha.
\]
Let $T:=\int_{\left(x-\eta,x\right)}\left(\nu\left(-z,z\right]\right)^{2}\nu\left(\textup{d}z\right)$.
Due to property 1 above, we have
\[
\left|T\right|=\left|B_{\nu}-\left(\nu\left(-x,x\right)+\alpha\right)^{2}\alpha\right|<\varepsilon.
\]

Then we calculate these terms:

\uline{\mbox{$A$}:} $A_{\nu}=A_{\pi_{\varepsilon}}$.

\uline{\mbox{$B$}:}
\begin{align*}
B_{\pi_{\varepsilon}} & =\int_{\left(x-\eta,x\right]}\left(\pi_{\varepsilon}\left(-z,z\right]\right)^{2}\pi_{\varepsilon}\left(\textup{d}z\right)=\int_{\left(x-\eta,x\right]}\left(\nu\left(-\left(x-\eta\right),x-\eta\right]+\frac{\alpha}{\eta}\cdot2\left(z-\left(x-\eta\right)\right)\right)^{2}\frac{\alpha}{\eta}\textup{d}z,
\end{align*}
where we used property 2. We set $y:=\nu\left(-\left(x-\eta\right),x-\eta\right]$.
By a change of variables $u:=y+\frac{2\alpha}{\eta}\left(z-\left(x-\eta\right)\right)$,
we obtain
\begin{align*}
B_{\pi_{\varepsilon}} & =\int_{y}^{y+2\alpha}u^{2}\cdot\frac{1}{2}\textup{d}z=\frac{1}{6}\left(\left(y+2\alpha\right)^{3}-y^{3}\right).
\end{align*}
The inequality $\nu\left(-x,x\right)-\nu\left(-\left(x-\eta\right),x-\eta\right]>-\varepsilon$
holds. We calculate bounds for $B_{\pi_{\varepsilon}}$ in terms of
$B_{\nu}$:
\begin{align*}
B_{\pi_{\varepsilon}}-B_{\nu} & =\frac{1}{6}\left(6\alpha y^{2}+12\alpha^{2}y+8\alpha^{3}\right)-T-\left(\left(\nu\left(-x,x\right)\right)^{2}+2\alpha\nu\left(-x,x\right)+\alpha^{2}\right)^{2}\alpha\\
 & =\alpha\left(y^{2}-\left(\nu\left(-x,x\right)\right)^{2}\right)+2\alpha^{2}\left(y-\nu\left(-x,x\right)\right)+\frac{1}{3}\alpha^{3}-T,
\end{align*}
so we obtain the bounds
\[
-\alpha\cdot2\varepsilon-2\alpha^{2}\varepsilon-\varepsilon<B_{\pi_{\varepsilon}}-\left(B_{\nu}+\frac{1}{3}\alpha^{3}\right)<\alpha\cdot2\varepsilon+2\alpha^{2}\varepsilon+\varepsilon,
\]
and hence
\[
-5\varepsilon<B_{\pi_{\varepsilon}}-\left(B_{\nu}+\frac{1}{3}\alpha^{3}\right)<5\varepsilon.
\]

\uline{\mbox{$C$}:} Since $\nu\left|\left(\left[-1/2,-x\right)\cup\left(x,1/2\right]\right)\right.$
is equal to $\pi_{\varepsilon}\left|\left(\left[-1/2,-x\right)\cup\left(x,1/2\right]\right)\right.$,
we have
\[
0\geq C_{\pi_{\varepsilon}}-C_{\nu}>-2\varepsilon\int_{\left(x,1/2\right]}\nu\left(\textup{d}z\right)\geq-2\varepsilon.
\]
In the second step above, we used that, for all $x,y\in\left[0,1\right]$,
$\left|x^{2}-y^{2}\right|<2\left|x-y\right|$ is satisfied.

Putting together the three parts, we obtain the lower bound for
\begin{align*}
a\left(\pi_{\varepsilon}\right)-a\left(\nu\right) & =2\left(A_{\pi_{\varepsilon}}-A_{\nu}+B_{\pi_{\varepsilon}}-B_{\nu}+C_{\pi_{\varepsilon}}-C_{\nu}\right)=2\left(B_{\pi_{\varepsilon}}-B_{\nu}\right)+2\left(C_{\pi_{\varepsilon}}-C_{\nu}\right)\\
 & >2\left(\frac{1}{3}\alpha^{3}-5\varepsilon\right)-2\varepsilon=\frac{2}{3}\alpha^{3}-12\varepsilon.
\end{align*}
Similarly, the upper bound is
\begin{align*}
a\left(\pi_{\varepsilon}\right)-a\left(\nu\right) & <2\left(\frac{1}{3}\alpha^{3}+5\varepsilon\right)=\frac{2}{3}\alpha^{3}+10\varepsilon.
\end{align*}

If we let $\varepsilon$ go to 0, we see that $a\left(\pi_{\varepsilon}\right)-a\left(\nu\right)$
goes to $2/3\alpha^{3}$. Hence removing a pair of points from the
discrete measure $\delta_{n}$ and adding the probability mass to
the continuous measure $\gamma$ increases $a$ as claimed.

We have shown that for any probability measure $\mu$ $a\left(\mu\right)\leq1/3$
holds. Since $r\geq0$ holds as can be seen from Lemma \ref{lem:mu_rho_a_r},
the term
\[
r-am=r-a\cdot2r
\]
is non-negative if and only if $r=0$ or $a\leq1/2$. The latter inequality
we have proved holds for all probability measures $\mu$. Corollary
\ref{cor:a_r_0} says that $r=0$ if and only if $\mu=\delta_{0}$.
For all other measures $\mu$, the optimal weights will be composed
of a constant and a proportional part.

\subsection*{Proof of Proposition \ref{prop:FOSD_suff}}

We use the well known characterisation of FOSD in terms of increasing
functions (usually referred to as utility functions in the context
of consumer theory in microeconomics):
\begin{lem}
\label{lem:char_FOSD}We have $\left|Z\right|\succ\left|Y\right|$
if and only if for all increasing functions $u:\left[0,1/2\right]\rightarrow\mathbb{R}$
the inequality $E_{\mu}u\geq E_{\rho}u$ holds.
\end{lem}

We employ the previous lemma to show
\begin{lem}
These two statements hold:
\begin{enumerate}
\item If for all $z\in\left(0,1/2\right]$ $\mu\left[-z,z\right]\leq\rho\left[-z,z\right]$,
then, for all $z\in\left[0,1/2\right]$, $\mu\left(-z,z\right)\leq\rho\left(-z,z\right)$.
\item If for all $z\in\left(0,1/2\right]$ $\mu\left[-z,z\right]\leq\rho\left[-z,z\right]$,
then, for all $z\in\left[0,1/2\right]$, $\mu\left(-z,z\right]\leq\rho\left(-z,z\right]$.
\end{enumerate}
\end{lem}

\begin{proof}
Let $z\in\left(0,1/2\right]$. Then, for all $t<z$, $\mu\left[-t,t\right]\leq\rho\left[-t,t\right]\leq\rho\left(-z,z\right)$.
By letting $t\nearrow z$, we obtain $\mu\left(-z,z\right)\leq\rho\left(-z,z\right)$
due to the continuity of the measure $\mu$, and we have proved the
first assertion. Next we show the second assertion:
\begin{align*}
\mu\left(-z,z\right] & =\mu\left(-z,0\right)+\mu\left\{ 0\right\} +\mu\left(0,z\right]\\
 & =\frac{\mu\left(-z,0\right)+\mu\left\{ 0\right\} +\mu\left(0,z\right)}{2}+\frac{\mu\left[-z,0\right)+\mu\left\{ 0\right\} +\mu\left(0,z\right]}{2}\\
 & =\frac{\mu\left(-z,z\right)}{2}+\frac{\mu\left[-z,z\right]}{2}\leq\frac{\rho\left(-z,z\right)}{2}+\frac{\rho\left[-z,z\right]}{2}=\rho\left(-z,z\right].
\end{align*}
We used the symmetry of $\mu$ and $\rho$ in steps 2 and 5 above
and the first assertion of the lemma in step 4.
\end{proof}
Now we calculate
\begin{align*}
r & =2\int_{\left(0,1/2\right]}z\rho\left(-z,z\right]\mu\left(\textup{d}z\right)\geq2\int_{\left(0,1/2\right]}y\rho\left(-y,y\right]\rho\left(\textup{d}y\right)\geq2\int_{\left(0,1/2\right]}y\mu\left(-y,y\right]\rho\left(\textup{d}y\right)=s.
\end{align*}
The first inequality is due to Lemma \ref{lem:char_FOSD}: The function
$z\mapsto u\left(z\right):=z\rho\left(-z,z\right]$ is increasing
and hence $E_{\mu}u\geq E_{\rho}u$ holds. The second inequality holds
by the definition of $\left|Z\right|\succ\left|Y\right|$. Therefore,
\[
am=a\left(r+s\right)\leq2ar,
\]
and $a\leq1/2$ is sufficient for $r-am\geq0$.

We turn the second statement of Proposition \ref{prop:FOSD_suff}.
Assume $\left|Y\right|\succ\left|Z\right|$. As we know from Theorem
\ref{thm:mu_rho},
\begin{align*}
a & =2\int_{\left(0,1/2\right]}\left(\rho\left(-z,z\right]\right)^{2}\mu\left(\textup{d}z\right)\leq2\int_{\left(0,1/2\right]}\left(\mu\left(-z,z\right]\right)^{2}\mu\left(\textup{d}z\right)\leq1/3.
\end{align*}
So a sufficient condition for $r-am\geq0$ is $m\leq3r$, which is
equivalent to $s\leq2r$.

\subsection*{Proof of Proposition \ref{prop:ribbon}}

The inequality $r-am\geq0$ we want to show is equivalent to $r\left(1-a\right)\geq as$.
The left hand side of this has a lower bound
\[
r\left(1-a\right)\geq\left(1-a\right)\cdot2\int_{\left(0,1/2\right]}z\rho\left(-z,z\right]c\rho\left(\textup{d}z\right),
\]
whereas the right hand side is bounded above by
\[
as\leq a\cdot2\int_{\left(0,1/2\right]}yC\rho\left(-y,y\right]\rho\left(\textup{d}y\right).
\]
So $c\left(1-a\right)\geq Ca$ is sufficient. This is itself equivalent
to
\begin{equation}
a\leq\frac{c}{c+C}.\label{eq:a_upper_bound_ribbon}
\end{equation}
We find two upper bounds for $a$:
\begin{align}
a & \leq C\cdot2\int_{\left(0,1/2\right]}\left(\rho\left(-z,z\right]\right)^{2}\rho\left(\textup{d}z\right),\label{eq:ribbon_1}\\
a & \leq\frac{1}{c^{2}}\cdot2\int_{\left(0,1/2\right]}\left(\mu\left(-z,z\right]\right)^{2}\mu\left(\textup{d}z\right).\label{eq:ribbon_2}
\end{align}
Theorem \ref{thm:mu_rho} says that both integrals on the right hand
side are bounded above by $1/6$. By stating inequalities of the right
hand side of (\ref{eq:a_upper_bound_ribbon}) and the right hand sides
of (\ref{eq:ribbon_1}) and (\ref{eq:ribbon_2}), respectively, we
obtain the sufficient conditions stated in Proposition \ref{prop:ribbon}:
\begin{align*}
\frac{C}{3}\leq\frac{c}{c+C} & \iff c\geq\frac{C^{2}}{3-C},\\
\frac{1}{3c^{2}}\leq\frac{c}{c+C} & \iff C\leq c\left(3c^{2}-1\right).
\end{align*}

The last claim follows from substituting $c=1/C$ into (\ref{eq:a_upper_bound_ribbon}).

\subsection*{Proof of Proposition \ref{prop:rho_unif_mu_any}}

We first refine Lemma \ref{lem:mu_neq_rho_a_r_s} using that $\rho=\mathcal{U}\left[-1/2,1/2\right]$:
\begin{align*}
a & =4E\left(Z^{2}\right),\quad r=2E\left(Z^{2}\right),\quad s=2\int_{0}^{1/2}y\mu\left(-y,y\right]\textup{d}y.
\end{align*}
So the inequality $r-am\geq0$ is equivalent to
\begin{align}
T\left(\mu\right):=\int_{0}^{1/2}y\mu\left(-y,y\right]\textup{d}y+E\left(Z^{2}\right)\leq\frac{1}{4}.\label{T_def}
\end{align}
The mapping $T:\mathcal{M}_{<\infty}\left(\left[-1/2,1/2\right]\right)\rightarrow\mathbb{R}$
from the set of all finite measures on $\left[-1/2,1/2\right]$ is
linear.
\begin{lem}
\label{lem:T_linear}For all $\nu_{1},\nu_{2}\in\mathcal{M}_{<\infty}\left(\left[-1/2,1/2\right]\right)$
and all $\alpha_{1},\alpha_{2}\in\mathbb{R}$, we have
\[
T\left(\alpha_{1}\nu_{1}+\alpha_{2}\nu_{2}\right)=\alpha_{1}T\left(\nu_{1}\right)+\alpha_{2}T\left(\nu_{2}\right).
\]
\end{lem}

This can be easily verified.

Our strategy to prove Proposition \ref{prop:rho_unif_mu_any} is to
show the result for discrete measures with finite support, and then
use the fact that discrete measures are a dense subset of $\mathcal{M}_{1}\left(\left[-1/2,1/2\right]\right)$.
The proof for discrete measures proceeds by induction on the size
of $\left|\textup{supp }\mu\right|\leq2k+1$.

\uline{Base case:} Let $k=1$. Then the support of $\mu$ consists
of at most three points: 0 and two points $-x_{1},x_{1}$ such that
$0<x_{1}\leq1/2$. The measure $\mu$ is given by $\beta_{0}\delta_{0}+\beta_{1}\left(\delta_{-x_{1}}+\delta_{x_{1}}\right)$
and the constants satisfy $\beta_{0}+2\beta_{1}=1$. Set $\beta:=\beta_{0}$.
We calculate
\begin{align*}
T\left(\mu\right) & =\int_{0}^{1/2}y\mu\left(-y,y\right]\textup{d}y+E\left(Z^{2}\right)=\frac{1}{8}+\frac{x_{1}^{2}}{2}\left(1-\beta\right).
\end{align*}
We see that we can choose the parameter $\beta$ and the point $x_{1}$
independently of each other to maximise $T\left(\mu\right)$. This
maximum is $1/4$ and it is achieved if and only if $\beta=0$ and
$x_{1}=1/2$. This shows the claim for $\left|\textup{supp }\mu\right|\leq2\cdot1+1$.

\uline{Induction step:} Assume for all $\mu\in\mathcal{M}_{1}\left(\left[-1/2,1/2\right]\right),\left|\textup{supp }\mu\right|\leq2k+1$,
the inequality $T\left(\mu\right)\leq1/4$ holds and equality is achieved
if and only if $\mu=1/2\left(\delta_{-1/2}+\delta_{1/2}\right)$.
We show that the claim also holds for all $\nu\in\mathcal{M}_{1}\left(\left[-1/2,1/2\right]\right)$
with $\left|\textup{supp }\nu\right|\leq2\left(k+1\right)+1$. Let
$0<x_{1}<\cdots<x_{k+1}\leq1/2$ be the points of the support of $\nu$.
Then $\nu$ must have the form
\[
\nu=\left(1-\eta\right)\mu+\eta\frac{1}{2}\left(\delta_{-x_{k+1}}+\delta_{x_{k+1}}\right)
\]
for some $0\leq\eta\leq1$. By Lemma \ref{lem:T_linear},
\begin{align*}
T\left(\nu\right) & =\left(1-\eta\right)T\left(\mu\right)+\eta T\left(\frac{1}{2}\left(\delta_{-x_{k+1}}+\delta_{x_{k+1}}\right)\right)\leq\left(1-\eta\right)\frac{1}{4}+\eta\frac{1}{4}=\frac{1}{4}.
\end{align*}
Furthermore, as $\left|\textup{supp }\mu\right|$ and $\left|\textup{supp }\frac{1}{2}\left(\delta_{-x_{k+1}}+\delta_{x_{k+1}}\right)\right|$
are at most $2k+1$ and $1/2\notin\textup{supp }\mu$, equality holds
if and only if $x_{k+1}=1/2$ and $\eta=1$. Hence, the second part
of the claim holds for $\left|\textup{supp }\nu\right|\leq2\left(k+1\right)+1$,
too.

A well known result concerning probability measures is
\begin{thm}
Let $X$ be a separable metric space. Then the set of discrete probability
measures on $X$ is dense in $\mathcal{M}_{1}\left(X\right)$ if we
consider $\mathcal{M}_{1}\left(X\right)$ as a space endowed with
the topology of weak convergence.
\end{thm}

See e.g.$\!\,$ Theorem 6.3 on page 44 in \cite{Pa}. Note that we can even
choose the subset of discrete probability measures with finite support
as a dense subset of $\mathcal{M}_{1}\left(X\right)$. We will now
show that the mapping $T$ is continuous. Let $\left(\mu_{n}\right)$
be a sequence in $\mathcal{M}_{1}\left(\left[-1/2,1/2\right]\right)$
with the limit $\mu\in\mathcal{M}_{1}\left(\left[-1/2,1/2\right]\right)$,
i.e. $\mu_{n}\xrightarrow[n\rightarrow\infty]{w}\mu$. We show that
both summands in the definition \eqref{T_def} of $T\left(\mu_n\right)$ converge.

The sequence of functions $y\mapsto y\mu_{n}\left(-y,y\right]$ is
uniformly bounded in $n$. $\mu_{n}\xrightarrow[n\rightarrow\infty]{w}\mu$
is equivalent to the convergence of the distribution functions. Let
$F_{n}$ be the distribution of $\mu_{n}$ for each $n$ and $F$
the distribution function of $\mu$. Then $F_{n}$ converges to $F$
pointwise on the set $C$ of continuity points of $F$. As $F$ is
monotonic, the complement $C^{c}$ is at most countable and hence
a Lebesgue null set. This means $y\mapsto y\mu_{n}\left(-y,y\right]$
converges almost everywhere on $\left(0,1/2\right]$. By dominated
convergence, the integrals $\int_{0}^{1/2}y\mu_{n}\left(-y,y\right]\textup{d}y$
converge to $\int_{0}^{1/2}y\mu\left(-y,y\right]\textup{d}y$.

The function $z\mapsto z^{2}$ is continuous and bounded on $\left(0,1/2\right]$.
Hence, $\mu_{n}\xrightarrow[n\rightarrow\infty]{w}\mu$ by definition
implies the convergence of $2\int_{\left(0,1/2\right]}z^{2}\mu_{n}\left(\textup{d}z\right)$
to $2\int_{\left(0,1/2\right]}z^{2}\mu\left(\textup{d}z\right)$.

We have previously shown that for all measures $\mu$ with finite
support $T\left(\mu\right)\leq1/4$. If $\mu$ is now any probability
measure, then there is a sequence of finitely supported measures $\mu_{n}$
that converge to $\mu$. As $T$ is continuous, this implies $T\left(\mu\right)\leq1/4$,
and the claim has been proved.

\subsection*{Proof of Theorem \ref{thm:fam_rho_mu}}

We first prove the four claims in Lemma \ref{lem:g_F_properties}
only assuming property 2 in Assumptions \ref{def:mu_contraction}.

\uline{Claim 1:} For all $y\in\left[0,1\right]$ $g\left(y\right)=2F_{\rho}\left(y/2\right)$.

Let $y\in\left[0,1\right]$. Since $F_{\rho}\left(\frac{1}{2}\right)=\frac{1}{2}$,
\[
\frac{g\left(y\right)}{2}=g\left(y\right)F_{\rho}\left(\frac{1}{2}\right)=F_{\rho}\left(\frac{y}{2}\right).
\]

\uline{Claim 2:} $\rho$ has no atoms, unless $\rho=\delta_{0}$.

We will write $f\left(x+\right)$ for the right limit $\lim_{t\searrow x}f\left(t\right)$
and $f\left(x-\right)$ for the left limit $\lim_{t\nearrow x}f\left(t\right)$
of any function $f$ and any $x\in\mathbb{R}$. Suppose $x>0$ is
an atom of $\rho$: $\rho\left\{ x\right\} >0$. Then $F_{\rho}\left(x-\right)<F_{\rho}\left(x\right)$.
Hence, for all $c<1$ $F_{\rho}\left(cx\right)=g\left(c\right)F_{\rho}\left(x\right)$.
Letting $c\nearrow1$, we get $F_{\rho}\left(x-\right)=g\left(1-\right)F_{\rho}\left(x\right)$,
so $0<g\left(1-\right)<1$. Thus we have, for all $y>0$, $F_{\rho}\left(y-\right)=g\left(1-\right)F_{\rho}\left(y\right)$
and $F_{\rho}\left(y-\right)<F_{\rho}\left(y\right)$, and $y$ is
an atom. This is a contradiction, since $\rho$ cannot have uncountably
many atoms. Therefore, $x>0$ cannot be an atom of $\rho$ and the
only possible atom is $0$. We next show that if $\rho\left\{ 0\right\} >0$,
then $\rho\left\{ 0\right\} =1$.

Suppose $0<\eta<1$ and $\rho\left|\left[0,1/2\right]\right.=\eta\delta_{0}+\frac{1-\eta}{2}\nu$,
where $\nu\in\mathcal{M}_{\leq1}\left(\left[0,1/2\right]\right)$
has no atoms. As $\nu$ has no atoms,
\[
\lim_{c\searrow0}F_{\rho}\left(cy\right)=F_{\rho}\left(0\right)
\]
holds for all $y\geq0$. On the other hand,
\[
\lim_{c\searrow0}g\left(c\right)F_{\rho}\left(y\right)=g\left(0+\right)F_{\rho}\left(y\right).
\]
Suppose $p:=g\left(0+\right)>0$. Fix some $b\in\left(0,1\right)$ such that $0<F_{\rho}\left(b\right)<\frac{1}{2}$. 
Then $p\left(F_{\rho}\left(\frac{1}{2}\right)-F_{\rho}\left(b\right)\right)>0$ and
\[
\lim_{c\searrow0}F_{\rho}\left(c\cdot\frac{1}{2}\right)=\lim_{c\searrow0}F_{\rho}\left(bc\right)=F_{\rho}\left(0\right)
\]
due to the right continuity of the distribution function $F_{\rho}$.
This implies $F_{\rho}\left(\frac{c}{2}\right)-F_{\rho}\left(bc\right)\rightarrow0$
as $c\searrow0$. But $g\left(c\right)\left(F_{\rho}\left(\frac{1}{2}\right)-F_{\rho}\left(b\right)\right)\rightarrow p>0$.
This is a contradiction and $g\left(0+\right)>0$ must be false. By
the first statement of this lemma, $F_{\rho}\left(0\right)=1/2\,g\left(0\right)\leq1/2\,g\left(0+\right)=0$.
The inequality is due to $g$ being increasing.

\uline{Claim 3:} $g$ is multiplicative: for all $x,y\geq0$, $g\left(xy\right)=g\left(x\right)g\left(y\right)$.

Let $x,y\geq0$. Then we have
\[
\frac{g\left(xy\right)}{2}=g\left(xy\right)F_{\rho}\left(\frac{1}{2}\right)=F_{\rho}\left(\frac{xy}{2}\right)=g\left(x\right)g\left(y\right)F_{\rho}\left(\frac{1}{2}\right)=\frac{g\left(x\right)g\left(y\right)}{2}.
\]

\uline{Claim 4:} $F_{\rho}$ has the form $F_{\rho}\left(y\right)=2^{t-1}y^{t}$
for some fixed $t\geq0$.

Since $g$ is multiplicative by statement 3, the transformation $x\mapsto f\left(x\right):=\ln\left(g\left(e^{x}\right)\right)$
is additive. Due to statement 1, $g$ is increasing. Hence, we can
apply the Cauchy functional condition to conclude that $f$ is linear,
i.e.$\!\,$ there is some $t\in\mathbb{R}$ such that $f\left(x\right)=tx$
for all $x\geq0$. As $f$ is increasing, $t$ must be non-negative.
So
\begin{align*}
\ln\left(g\left(e^{x}\right)\right) & =tx\iff g\left(e^{x}\right)=e^{tx}
\end{align*}
and
\[
g\left(x\right)=g\left(e^{\ln x}\right)=\left(e^{\ln x}\right)^{t}=x^{t}.
\]
Using statement 1, we obtain, for all $y\in\left[0,1/2\right]$,
\[
F_{\rho}\left(y\right)=\frac{g\left(2y\right)}{2}=\frac{\left(2y\right)^{t}}{2}=2^{t-1}y^{t}.
\]
From now on, we assume all three properties in Assumptions \ref{def:mu_contraction}
and show Theorem \ref{thm:fam_rho_mu}. First, we note that the homogeneity
property 2 of the measure $\rho$ is inherited by $\mu$:
\begin{lem}
For all $x\geq0$ and all $y\in\left[0,1/2\right]$ such that $xy\leq1/2$,
we have $\mu\left(0,xy\right)=g\left(x\right)\mu\left(0,y\right)$.
\end{lem}
The proof is straightforward and we thus omit it.

We next calculate an inequality equivalent to $r-am\geq0$:
\begin{align*}
r\left(1-a\right)\geq as & \quad\iff\\
2\int_{\left(0,1/2\right]}z\rho\left(-z,z\right]\mu\left(\textup{d}z\right)\left[1-2\int_{\left(0,1/2\right]}\left(\rho\left(-z,z\right]\right)^{2}\mu\left(\textup{d}z\right)\right]\geq\\
2\int_{\left(0,1/2\right]}\left(\rho\left(-z,z\right]\right)^{2}\mu\left(\textup{d}z\right)\cdot2\int_{\left(0,1/2\right]}y\mu\left(-y,y\right]\rho\left(\textup{d}y\right) & \quad\iff\\
\int_{\left(0,c/2\right)}z\rho\left(0,z\right)\mu\left(\textup{d}z\right)\left[1-8\int_{\left(0,1/2\right)}\left(\mu\left(0,cz\right)\right)^{2}\mu\left(\textup{d}z\right)\right]\geq\\
8\int_{\left(0,1/2\right)}\left(\mu\left(0,cz\right)\right)^{2}\mu\left(\textup{d}z\right)\left[\int_{\left(0,c/2\right)}y\mu\left(0,y\right)\rho\left(\textup{d}y\right)+\frac{1}{2}\int_{\left(c/2,1/2\right)}y\rho\left(\textup{d}y\right)\right],
\end{align*}
where we used the symmetry of $\rho$ and $\mu$, the fact that $\rho$
-- and hence $\mu$ -- has no atoms, and $\mu\left(0,cz\right)=\rho\left(0,z\right)$.
The left hand side of the last inequality above can be expressed as
\[
\int_{\left(0,c/2\right)}z\rho\left(0,z\right)\mu\left(\textup{d}z\right)\left[1-\frac{g\left(c\right)^{2}}{3}\right],
\]
where we applied Theorem \ref{thm:mu_rho}. The right hand side can
be treated similarly:
\[
\frac{g\left(c\right)^{2}}{3}\left[\int_{\left(0,c/2\right)}y\mu\left(0,y\right)\rho\left(\textup{d}y\right)+\frac{1}{2}\int_{\left(c/2,1/2\right)}y\rho\left(\textup{d}y\right)\right].
\]
We note that the inequality
\begin{align}
\left[3-g\left(c\right)^{2}\right]\int_{\left(0,c/2\right)}z\rho\left(0,z\right)\mu\left(\textup{d}z\right) & \geq g\left(c\right)^{2}\left[\int_{\left(0,c/2\right)}y\mu\left(0,y\right)\rho\left(\textup{d}y\right)+\frac{1}{2}\int_{\left(c/2,1/2\right)}y\rho\left(\textup{d}y\right)\right]\label{eq:r_am_equiv_1}
\end{align}
is thus equivalent to our original inequality. Now we show
\begin{lem}
We can switch the measures in the integrals as follows:
\[
\int_{\left(0,c/2\right)}z\rho\left(0,z\right)\mu\left(\textup{d}z\right)=\int_{\left(0,c/2\right)}y\mu\left(0,y\right)\rho\left(\textup{d}y\right).
\]
\end{lem}

\begin{proof}
The proof uses the Lebesgue-Stieltjes versions of the integrals. Let
$F_{\mu}$ be the distribution function of $\mu\left|\left[0,1/2\right]\right.$.
We have
\begin{align*}
\int_{\left(0,c/2\right)}z\rho\left(0,z\right)\mu\left(\textup{d}z\right) & =\int_{\left(0,c/2\right)}z\mu\left(0,cz\right)\mu\left(\textup{d}z\right)=g\left(c\right)\int_{\left(0,c/2\right)}z\mu\left(0,z\right)\textup{d}F_{\mu}\left(z\right)\\
 & =\int_{\left(0,c/2\right)}z\mu\left(0,z\right)\textup{d}\left(g\left(c\right)F_{\mu}\left(z\right)\right)=\int_{\left(0,c/2\right)}z\mu\left(0,z\right)\textup{d}F_{\mu}\left(cz\right)\\
 & =\int_{\left(0,c/2\right)}z\mu\left(0,z\right)\textup{d}F_{\rho}\left(z\right)=\int_{\left(0,c/2\right)}y\mu\left(0,y\right)\rho\left(\textup{d}y\right)
\end{align*}
by a substitution formula (see e.g. \cite{FaTe}).
\end{proof}
Using this lemma, we can restate (\ref{eq:r_am_equiv_1}) as
\begin{equation}
\left[3-2g\left(c\right)^{2}\right]\int_{\left(0,c/2\right)}z\rho\left(0,z\right)\mu\left(\textup{d}z\right)\geq\frac{g\left(c\right)^{2}}{2}\int_{\left(c/2,1/2\right)}y\rho\left(\textup{d}y\right).\label{eq:r_am_equiv_2}
\end{equation}
Next we prove
\begin{lem}
The following equality holds:
\[
\frac{1}{g\left(c\right)}\int_{\left(0,c/2\right)}z\rho\left(0,z\right)\mu\left(\textup{d}z\right)=c\int_{\left(0,1/2\right)}y\rho\left(0,y\right)\rho\left(\textup{d}y\right).
\]
\end{lem}

\begin{proof}
We calculate
\begin{align*}
\int_{\left(0,1/2\right)}cz\rho\left(0,z\right)\rho\left(\textup{d}z\right) & =\int_{\left(0,1/2\right)}\frac{1}{g\left(c\right)}g\left(c\right)cy\rho\left(0,y\right)\textup{d}F_{\rho}\left(y\right)=\frac{1}{g\left(c\right)}\int_{\left(0,c/2\right)}g\left(c\right)x\rho\left(0,\frac{x}{c}\right)\textup{d}F_{\rho}\left(\frac{x}{c}\right)\\
 & =\frac{1}{g\left(c\right)}\int_{\left(0,c/2\right)}x\rho\left(0,x\right)\textup{d}F_{\mu}\left(x\right).
\end{align*}
\end{proof}
Together the last lemma and statement 4 of Lemma \ref{lem:g_F_properties}
imply the inequality (\ref{eq:r_am_equiv_2}) is equivalent to
\[
c^{1-t}\left[3-2c^{2t}\right]\geq\frac{1+2t}{1+t}\left(1-c^{1+t}\right).
\]
We define the function $h$ by setting
\[
h\left(c\right):=c^{1+t}-3\left(1+t\right)c^{1-t}+1+2t.
\]
The original inequality holds if and only if $h\left(c\right)\leq0$.
We calculate the first derivative of $h$,
\[
h'\left(c\right)=\left(1+t\right)c^{t}-3\left(1+t\right)\left(1-t\right)c^{-t}
\]
and the critical point is given by
\begin{equation}
c=x_{0}:=\left(3\left(1-t\right)\right)^{\frac{1}{2t}}\label{eq:x_0},
\end{equation}
which is positive if $t<1$. The second derivative of $h$ is
\[
h''\left(c\right)=\left(1+t\right)tc^{t-1}+3\left(1+t\right)\left(1-t\right)tc^{-\left(1+t\right)}.
\]
The sign of the second derivative is positive for all $c>0$. We also
note that $h\left(0\right)>0$ and $h\left(1\right)<0$. The positive
sign of $h''$ on $\left(0,\infty\right)$ implies that $h'$ is strictly
increasing on $\left[0,\infty\right)$. Furthermore, $h$ is strictly
decreasing on $\left[0,x_{0}\right)$ and strictly increasing on $\left[x_{0},\infty\right)$.
It is clear from (\ref{eq:x_0}) that $x_{0}\geq1$ if and only if
$t\leq2/3$. When this holds, $h$ is strictly decreasing on $\left[0,1\right]$.
For $2/3<t<1$, $h$ is first decreasing and then increasing. However,
as $h\left(1\right)<0$, we have for all $t<1$ a uniquely determined
$c_{0}\in\left(0,x_{0}\wedge1\right)$ such that $h\left(c_{0}\right)=0$
and $c_{0}$ is the only zero of $h$ on the interval $\left[0,1\right]$.
For $t=1$, the claim follows from Proposition \ref{prop:rho_unif_mu_any}.
For $t>1$, we note that $h$ is undefined at $c=0$ but $\lim_{c\searrow0}h\left(c\right)=-\infty$.
As $h\left(1\right)<0$ holds for any value of $t$ and $h$ is continuous,
$h<0$ on $\left(0,1\right)$ is clear. This shows the claim concerning
the sign of $r-am$.

As for the behaviour of the critical $c_{0}$ at which $r=am$, by
inspecting (\ref{eq:x_0}), we see that $\lim_{t\nearrow1}x_{0}=0$
and as $0<c_{0}<x_{0}$, the second claim follows.

Contact information of the authors:

werner.kirsch@fernuni-hagen.de

gabor.toth@iimas.unam.mx

\end{document}